\documentclass[a4paper]{article}

\usepackage[english]{babel}
\usepackage[utf8x]{inputenc}
\usepackage[T1]{fontenc}
\newcommand{\nm}{\noalign{\smallskip}}
\newcommand{\ds}{\displaystyle}
\usepackage{graphicx}
\usepackage{subcaption}
\usepackage{tikz}
\usepackage{pgfplots}
\usepackage{wrapfig}
\usepackage{cutwin}
\graphicspath{ {MATLAB/linedefect/} }
\usetikzlibrary{intersections,decorations.pathreplacing,decorations.markings,calc}
\usepackage[a4paper,top=3cm,bottom=2cm,left=2.5cm,right=2.5cm,marginparwidth=1.75cm]{geometry}

\usepackage{amsmath}
\usepackage{amsthm}
\usepackage{amssymb}
\makeatletter
\newcommand{\neutralize}[1]{\expandafter\let\csname c@#1\endcsname\count@}
\makeatother
\newtheorem{prop}{Proposition}[section]

\newtheorem{thm}{Theorem}[section]
\newtheorem{lem}{Lemma}[section]

\newtheorem{assump}{Assumption}[section]
\theoremstyle{definition}
\newtheorem{rmk}{Remark}[section]
\numberwithin{equation}{section}
\usepackage{tikz}
\usepackage{graphicx}
\usepackage{hyperref}
\usepackage{caption}
\usepackage{float}
\theoremstyle{definition}
\newcommand{\D}{\mathcal{D}}
\newcommand{\Z}{\mathbb{Z}}
\newcommand{\R}{\mathbb{R}}
\newcommand{\C}{\mathcal{C}}
\newcommand{\A}{\mathcal{A}}

\newcommand{\F}{\mathcal{F}}
\newcommand{\p}{\partial}
\renewcommand{\L}{\mathcal{L}}
\renewcommand{\S}{\mathcal{S}}
\newcommand{\M}{\mathcal{M}}
\newcommand{\K}{\mathcal{K}}

\renewcommand{\epsilon}{\varepsilon}
\newcommand{\dx}{\: \mathrm{d}}

\newcommand{\eqnref}[1]{(\ref {#1})}
\def\nm{\noalign{\medskip}}
\def\capacity{{\mathrm{Cap}}}

\newcommand{\ie}{\textit{i.e.}}

\title{Subwavelength guided modes for acoustic waves in bubbly crystals with a line defect}

\author{
	Habib Ammari\thanks{\footnotesize Department of Mathematics, 
		ETH Z\"urich, 
		R\"amistrasse 101, CH-8092 Z\"urich, Switzerland (habib.ammari@math.ethz.ch, erik.orvehed.hiltunen@sam.math.ethz.ch, sanghyeon.yu@sam.math.ethz.ch).}\and Erik Orvehed Hiltunen\footnotemark[1]  \and Sanghyeon Yu\footnotemark[1]}

\date{}

\begin{document}
	\maketitle

\begin{abstract}
The recent development of subwavelength photonic and phononic crystals shows the possibility of controlling wave propagation at deep subwavelength scales. Subwavelength bandgap phononic crystals are typically created using a periodic arrangement of subwavelength resonators, in our case small gas bubbles in a liquid. In this work, a waveguide is created by modifying the sizes of the bubbles along a line in a dilute two-dimensional bubbly crystal, thereby creating a line defect. Our aim  is to prove that the line defect indeed acts as a waveguide; waves of certain frequencies will be localized to, and guided along, the line defect. The key result is an original formula for the frequencies of the defect modes. Moreover, these frequencies are numerically computed using the multipole method, which numerically illustrates our main results.
\end{abstract}

\def\keywords2{\vspace{.5em}{\textbf{  Mathematics Subject Classification
(MSC2000).}~\,\relax}}
\def\endkeywords2{\par}
\keywords2{35R30, 35C20.}

\def\keywords{\vspace{.5em}{\textbf{ Keywords.}~\,\relax}}
\def\endkeywords{\par}
\keywords{bubble, subwavelength resonance, subwavelength phononic crystal,  subwavelength wave\-guide, line defect, weak localization.}

\section{Introduction}	
Line defects in bandgap photonic or phononic bandgap crystals are of interest due to their possible applications in low-loss waveguides. The main mathematical problem of interest is to show that the spectrum of the defect operator has a non-zero overlap with the original bandgap. Moreover,  it is also of interest to understand the nature and location of the defect spectrum. For previous works regarding line defects in bandgap crystals we refer to \cite{santosa,kuchment_line, kuchment_EM, Brown, Brown1, cardone1, Delourme_dilute, Fliss_DtN}.

In this work, we consider a line defect in a phononic bandgap crystal comprised of gas bubbles in a liquid. The gas bubbles are known to resonate at a low frequency, called the \textit{Minnaert frequency}. The corresponding wavelength is larger than the bubble by several orders of magnitude \cite{first, minnaert}. Based on this, it is possible to create \textit{subwavelength} bandgap crystals, which operate at wavelengths much larger than the unit cell size of the microstructured material. One of the main motivations for studying subwavelength bandgap materials is to manipulate wave propagation at subwavelength scales. A second motivation  is for their use in devices where conventional bandgap materials, based on Bragg scattering, would create infeasibly large devices \cite{pnas, nature}. Mathematical properties of bubbly phononic bandgap materials have been studied in, for example, \cite{first, defectSIAM, doublenegative, bandgap, nearzero, effectivemedium}, and subwavelength phononic bandgap materials have been experimentally realised in \cite{experiment2013, Lemoult_sodacan, phononic1}.

Wave localization due to a point defect in a bubbly bandgap material was first proven in \cite{defectSIAM}. In \cite{defectX}, where some additions and minor corrections to \cite{defectSIAM} were made,  it is shown  that the mechanism for creating localized modes using small perturbations is quite different depending on the volume fraction of the bubbles. In order to create localized modes in the dilute regime, the defect should be smaller than the surrounding bubbles, while in the non-dilute regime, the defect has to be larger. Based on this, in the case of a line defect, it is natural to expect different behaviour in these two different regimes. This suggests that different methods of analysis are needed in the two regimes. In this paper, we will mainly focus on the dilute regime, taking the radius of the bubbles sufficiently small.

If the defect size is small, \ie{} if the size of the
perturbed bubble is close to its original size, then the band structure of the defect problem will be a small perturbation of the band structure of the original problem \cite{MaCMiPaP, AKL}. This way, it is possible to shift the defect band upwards, and a part of the defect band will fall into the subwavelength bandgap. However, because of the curvature of the original band, it is impossible to create a defect band entirely inside the bandgap with this approach.

In order to create defect bands which are entirely located inside the subwavelength bandgap, we have to consider slightly larger perturbations. In this paper, we will show that for arbitrarily small defects, a part of the defect band will lie inside the bandgap. Moreover, we will show that for suitably large perturbation sizes, the entire defect band will fall into the bandgap, and we will explicitly quantify the size of the perturbation needed in order to achieve this. Because of this, our results are more general than previous weak localization results since we explicitly show how the defect band depends on the perturbation size. 

In order to have \textit{guided} waves along the line defect, the defect mode must not only be localized to the line, but also propagating along the line. In other words, we must exclude the case of standing waves in the line defect, \ie{} modes which are localized in the direction of the line. As discussed in \cite{kuchment_review,kuchment_line}, such modes are associated with the point spectrum of the perturbed operator which appears as a flat band in the dispersion relation. Proving the absence of bound modes in phononic or photonic waveguides is a challenging problem;  for example in \cite{hardwall1} this was proven by imposing ``hard-wall'' Dirichlet or Neumann boundary conditions along the waveguide, while in \cite{surface4} the absence of bound modes was proven in the case of a simpler Helmholtz-type operator. In this paper, we use the explicit formula for the defect band to show that it is nowhere flat, and hence does not correspond to bound modes in the direction of the line.  

The paper is structured as follows. In Section \ref{sec-1} we discuss preliminary results on layer potentials, and outline the main results from \cite{bandgap}. In Section \ref{sec-2} we restrict to circular domains and follow the approach of \cite{defectSIAM,defectX} to model the line defect using the fictitious source superposition method, originally introduced in \cite{Wilcox}. In Section \ref{sec-3} we prove the existence of a defect resonance frequency, and derive an asymptotic formula in terms of the density contrast in the dilute regime. Using this formula, we show that the defect modes are localized to, and guided along, the line defect. In Section \ref{sec:num} we compute the defect band numerically, in order to verify the formula and also illustrate the behaviour in the non-dilute regime. The paper ends with some concluding remarks in Section \ref{sec-5}.  In Appendix \ref{app:nondil}, we restrict ourselves to small perturbations to derive an asymptotic formula valid in the non-dilute regime. In Appendix \ref{app:noncirc} we outline the fictitious source superposition method in the case of non-circular domains.

\section{Preliminaries} \label{sec-1}
\subsection{Layer potentials} \label{sec:layerpot}
Let $Y^2 =[-1/2,1/2)^2\subset \R^2$ be the unit cell and assume that the bubble occupies a bounded and simply connected domain $D\in Y^2$ with $\p D \in C^{1,s}$ for some $0<s<1$. Let $\Gamma^0$ and $\Gamma^k,k>0$  be the Green's functions of the Laplace and Helmholtz equations in dimension two, respectively, \ie{}, 
\begin{equation*}
\begin{cases}
\ds \Gamma^k(x,y) = -\frac{i}{4}H_0^{(1)}(k|x-y|), \ & k>0, \\
\nm
\ds \Gamma^0(x,y) = \frac{1}{2\pi}\ln|x-y|, & k=0,
\end{cases}
\end{equation*}
where $H_0^{(1)}$ is the Hankel function of the first kind and order zero. Here, the outgoing Sommerfeld radiation condition is used for selecting the physical Helmholtz Green's function \cite{MaCMiPaP}.

Let $\S_{D}^k: L^2(\partial D) \rightarrow H_{\textrm{loc}}^1(\R^2)$ be the single layer potential defined by
\begin{equation*}
\S_D^k[\phi](x) = \int_{\partial D} \Gamma^k(x,y)\phi(y) \dx \sigma(y), \quad x \in \R^2.
\end{equation*}
Here, $H_{\textrm{loc}}^1(\R^2)$  denotes the space of functions that, on every compact subset of $\R^2$, are square integrable and have a weak first derivative that is also square integrable.

We also define the Neumann-Poincar\'e operator $\K_D^{k,*}: L^2(\partial D) \rightarrow L^2(\partial D)$ by
\begin{equation*}
\K_D^{k,*}[\phi](x) = \int_{\partial D} \frac{\partial }{\partial \nu_x}\Gamma^k(x,y) \phi(y) \dx \sigma(y), \quad x \in \partial D.
\end{equation*}
The following so-called jump relations of $\S_D^k$ on the boundary $\partial D$ are well-known (see, for example, \cite{MaCMiPaP}):
\begin{equation*}
\S_D^k[\phi]\big|_+ = \S_D^k[\phi]\big|_-,
\end{equation*}
and
\begin{equation*}
\frac{\partial }{\partial \nu}\S_D^k[\phi] \bigg|_{\pm} =  \left(\pm\frac{1}{2} I + \K_D^{k,*}\right) [\phi].
\end{equation*}
Here, $\partial/\partial \nu$ denotes the outward normal derivative,  and $|_\pm$ denote the limits from outside and inside $D$. In two dimensions, we have the following expansion of the Green's function for the Helmholtz equation \cite{MaCMiPaP}
\begin{equation*}\label{eq:hankel}
-\frac{i}{4}H_0(k|x-y|) = \frac{1}{2\pi} \ln |x-y| + \eta_k + \sum_{j=1}^\infty\left( b_j \ln(k|x-y|) + c_j \right) (k|x-y| )^{2j},
\end{equation*}
where $\ln$ is the principal branch of the logarithm and
$$ \eta_k = \frac{1}{2\pi}(\ln k+\gamma-\ln 2)-\frac{i}{4}, \quad b_j=\frac{(-1)^j}{2\pi}\frac{1}{2^{2j}(j!)^2}, \quad c_j=b_j\left( \gamma - \ln 2 - \frac{i\pi}{2} - \sum_{n=1}^j \frac{1}{n} \right),$$
with $\gamma$ being the Euler constant. Define, for $\phi \in L^2(\partial D)$, 
\begin{equation*}
\hat{S}_D^k[\phi](x) = \S_D[\phi](x) + \eta_k\int_{\partial D} \phi\dx \sigma.
\end{equation*}
Then the following expansion holds:
\begin{equation} \label{eq:Sexpansion}
\S_D^k =  \hat{\S}_{D}^k +O(k^2\ln k). 
\end{equation}

We also introduce a quasi-periodic version of the layer potentials. For $\alpha\in [0,2\pi)^2$, the quasi-periodic Green's function $\Gamma^{\alpha, k}$ is defined to satisfy
$$ (\Delta_x + k^2) \Gamma^{\alpha, k} (x,y) = \sum_{n\in \mathbb{R}^2} \delta(x-y-n) e^{i n\cdot \alpha}, \qquad x,y\in Y,$$
where $\delta$ is the Dirac delta function. The function $\Gamma^{\alpha, k} $ is $\alpha$-quasi-periodic in $x$, \ie{},  $e^{- i \alpha\cdot x} \Gamma^{\alpha, k}(x,y)$ is periodic in $x$ with respect to $Y$.  

We define the quasi-periodic single layer potential $\mathcal{S}_D^{\alpha,k}$ by
$$\mathcal{S}_D^{\alpha,k}[\phi](x) = \int_{\partial D} \Gamma^{\alpha,k} (x,y) \phi(y) \dx\sigma(y),\quad x\in \mathbb{R}^2.$$
It satisfies the following jump formulas:
\begin{equation*}
\S_D^{\alpha,k}[\phi]\big|_+ = \S_D^{\alpha,k}[\phi]\big|_-,
\end{equation*}
and
$$ \frac{\p}{\p\nu} \Big|_{\pm} \mathcal{S}_D^{\alpha,k}[\phi] = \left( \pm \frac{1}{2} I +( \mathcal{K}_D^{-\alpha,k} )^*\right)[\phi]\quad \mbox{on}~ \p D,$$
where $(\mathcal{K}_D^{-\alpha,k})^*$ is the operator given by
$$ (\mathcal{K}_D^{-\alpha, k} )^*[\phi](x)= \int_{\p D} \frac{\p}{\p\nu_x} \Gamma^{\alpha,k}(y,y) \phi(y) \dx\sigma(y).$$
We recall that $\mathcal{S}_D^{\alpha,0} : L^2(\p D) \rightarrow H^1(\p D)$ is invertible for $\alpha \ne 0$ \cite{MaCMiPaP}.

\subsection{Floquet transform}
A function $f(x_1)$ is said to be $\alpha$-quasi-periodic in the variable $x_1\in \R$
if $e^{-i\alpha x_1}f(x_1)$ is periodic. Given a function $f\in L^2(\R)$, the Floquet transform in one dimension is defined as
\begin{equation}\label{eq:floquet}
\F[f](x_1,\alpha) = \sum_{m\in \Z} f(x_1-m) e^{i\alpha m},
\end{equation}
which is $\alpha$-quasi-periodic in $x_1$ and periodic in $\alpha$. Let $Y = [-1/2,1/2)$ be the unit cell and $Y^* := \R / 2\pi \Z \simeq [0,2\pi)$ be the Brillouin zone. The Floquet transform is an invertible map $\F:L^2(\R) \rightarrow L^2(Y\times Y^*)$, with inverse  (see, for instance, \cite{kuchment, MaCMiPaP})
\begin{equation*}
\F^{-1}[g](x_1) = \frac{1}{2\pi}\int_{Y^*} g(x_1,\alpha) \dx \alpha.
\end{equation*}

\subsection{Bubbly crystals and subwavelength bandgaps}\label{subsec:bandgap}

Here we briefly review the subwavelength bandgap opening of a bubbly crystal from \cite{bandgap}. 

Assume that a single bubble occupies the region $D$ specified in Section \ref{sec:layerpot}. We denote by $\rho_b$ and $\kappa_b$ the density and the bulk modulus inside the bubble, respectively. We let $\rho_w$ and $\kappa_w$ be the corresponding parameters outside the bubble.  We introduce
\begin{equation*} 
v_w = \sqrt{\frac{\kappa_w}{\rho_w}}, \quad v_b = \sqrt{\frac{\kappa_b}{\rho_b}}, \quad k_w= \frac{\omega}{v_w} \quad \text{and} \quad k_b= \frac{\omega}{v_b}
\end{equation*}
as the speed of sound outside and inside the bubbles, and the wavenumber outside and inside the bubbles, respectively. Here, $\omega$ corresponds to the operating frequency of the acoustic waves. Let $\C = \cup_{n\in\Z^2}(D+n)$ be the periodic bubbly crystal. Define, for $x\in \R^2$,
$$\rho(x) = \rho_b\chi_{\C}(x) + \rho_w(1-\chi_{\C}(x)), \quad \kappa(x) = \kappa_b\chi_{\C}(x) + \kappa_w(1-\chi_{\C}(x)),$$
where $\chi_{\C}$ is the characteristic function of $\C$. 

We assume that there is a large contrast in the density, that is, the density contrast $\delta$ satisfies
\begin{equation} \label{data2}
\delta = \frac{\rho_b}{\rho_w} \ll 1.
\end{equation}
Recall that under (\ref{data2}), there exists a subwavelength resonance of the bubble in free space \cite{first}. 

In the following, we shall also make the assumption stated below.
\begin{assump} \label{assumption1}
	Without loss of generality, we
assume that $$v_w = v_b = 1.$$
\end{assump}
In this case we have $k_b = k_w = \omega$. Assumption \ref{assumption1} only serves to simplify the expressions. The methods presented in this paper indeed apply as long as the wave speeds outside and inside the bubbles are comparable to each other.

The wave propagation problem inside the periodic crystal can be modelled as
\begin{equation}\label{eq:original} \kappa(x) \nabla \cdot \left(\frac{1}{\rho(x)} \nabla v(x) \right) + \omega^2 v(x) = 0, \quad x \in \R^2.\end{equation}
We  denote by $\Lambda_0$ the set of propagating frequencies, \ie{}, the set of $\omega$ such that $\omega^2$ is in the spectrum of the operator
$$-\kappa \nabla \cdot \frac{1}{\rho} \nabla.$$

Denote by $Y_s = Y\times \R$ the unit strip and recall that $Y^2 = [-1/2,1/2)^2$ is the unit cell of the crystal. Applying the Floquet transformation, first in $x_1$-direction and then in $x_2$-direction, equation \eqnref{eq:original} can be decomposed first as
\begin{equation}\label{eq:Fonce} \begin{cases}
\ds \kappa(x) \nabla \cdot \left(\frac{1}{\rho(x)} \nabla v(x) \right) + \omega^2 v(x) = 0, \quad x \in Y_s, \\ 
\ds e^{-i \alpha_1 x_1} u  \,\,\,  \mbox{is periodic in } x_1,
 \end{cases}
\end{equation}
where $\alpha_1 \in Y^*$, and then as 
\begin{equation}\label{eq-scattering-quasiperiodic} \begin{cases}
\ds \kappa(x) \nabla \cdot \left(\frac{1}{\rho(x)} \nabla v(x) \right) + \omega^2 v(x) = 0, \quad x \in Y^2, \\ 
\ds e^{-i \alpha \cdot x} u  \,\,\,  \mbox{is periodic in } x,
\end{cases}
\end{equation}
where $\alpha = (\alpha_1, \alpha_2) \in Y^*\times Y^*$. We  denote by $\Lambda_{0,\alpha_1}$ the set of $\omega$ such that $\omega^2$ is in the spectrum of the operator implied by \eqnref{eq:Fonce} and by $\Lambda^{ess}_{0,\alpha_1}$ the essential part of this spectrum. 
It is known that \eqref{eq-scattering-quasiperiodic} has non-trivial solutions for discrete values of $\omega$:
$$  0 \le \omega_1^\alpha \le \omega_2^\alpha \le \cdots,$$
and we have the following band structure of propagating frequencies for the periodic bubbly crystal $\C$:
\begin{align*}
\ds \Lambda_{0,\alpha_1} &= \left[\min_{\alpha_2\in Y^*} \omega_1^{(\alpha_1,\alpha_2)},\max_{\alpha_2\in Y^*} \omega_1^{(\alpha_1,\alpha_2)}\right] \cup \left[\min_{\alpha_2\in Y^*} \omega_2^{(\alpha_1,\alpha_2)}, \max_{\alpha_2\in Y^*} \omega_2^{(\alpha_1,\alpha_2)} \right] \cup \cdots, \\
\ds \Lambda_0 &= \left[0,\max_{\alpha\in Y^*\times Y^*} \omega_1^\alpha\right] \cup \left[\min_{\alpha\in Y^*\times Y^*} \omega_2^\alpha, \max_{\alpha\in Y^*\times Y^*} \omega_2^\alpha \right] \cup \cdots.
\end{align*}

In \cite{bandgap}, it is proved that there exists a subwavelength spectral gap opening in the band structure.  Let us briefly review this result. 
We look for a solution $v$ of  \eqref{eq-scattering-quasiperiodic} which has the following form:
\begin{equation*} \label{Helm-solution}
v =
\begin{cases}
\mathcal{S}_{D}^{\alpha,k_w} [\varphi^\alpha]\quad & \text{in} ~ Y^2 \setminus \overline{D},\\
 \mathcal{S}_{D}^{k_b} [\psi^\alpha]   &\text{in} ~   {D},
\end{cases}
\end{equation*}
for some densities $\varphi^\alpha, \psi^\alpha \in  L^2(\p D)$. 
Using the jump relations for the single layer potentials, one can show that~\eqref{eq-scattering-quasiperiodic} is equivalent to the boundary integral equation
\begin{equation}  \label{eq-boundary}
\mathcal{A}^\alpha(\omega, \delta)[\Phi^\alpha] =0,  
\end{equation}
where
\[
\mathcal{A}^\alpha(\omega, \delta) = 
 \begin{pmatrix}
  \mathcal{S}_D^{k_b} &  -\mathcal{S}_D^{\alpha,k}  \\
  -\frac{1}{2}+ \K_D^{k_b,*}& -\delta\left( \frac{1}{2}+ \left(\mathcal{K}_D^{ -\alpha,k}\right)^*\right)
\end{pmatrix}, 
\,\, \Phi^\alpha= 
\begin{pmatrix}
\varphi^\alpha\\
\psi^\alpha
\end{pmatrix}.
\]

Since it can be shown that $\omega=0$ is a characteristic value for the operator-valued analytic function $\mathcal{A}(\omega,0)$, we can conclude the following result by the Gohberg-Sigal theory \cite{MaCMiPaP, Gohberg1971}.
\begin{lem}
For any $\delta$ sufficiently small, there exists a characteristic value 
$\omega_1^\alpha= \omega_1^\alpha(\delta)$ to the operator-valued analytic function 
$\mathcal{A}^\alpha(\omega, \delta)$
such that $\omega_1^\alpha(0)=0$ and 
$\omega_1^\alpha$ depends on $\delta$ continuously.
\end{lem}

The next theorem gives the asymptotic expansion of $\omega_1^\alpha$ as $\delta\rightarrow 0$.

\begin{thm}{\rm{\cite{bandgap}}} \label{approx_thm} For $\alpha \ne 0$ and sufficiently small $\delta$, we have
	\begin{align*}
	\omega_1^\alpha= \sqrt{\frac{\delta \capacity_{D,\alpha}}{|D|}} + O(\delta^{3/2}), \label{o_1_alpha}
	\end{align*}
	where the constant $\capacity_{D,\alpha}$ is given by
	$$\capacity_{D,\alpha}:= - \langle(\mathcal{S}_D^{\alpha,0})^{-1} [\chi_{\partial D}], \chi_{\partial D}\rangle.$$
	Here,  $\langle \,\cdot \,,\,\cdot\, \rangle$ stands for the standard inner product of $L^2(\partial D)$ and $\chi_{\partial D}$ denotes the characteristic function of $\partial D$.
\end{thm}

Let $\omega_1^*=\max_\alpha \omega_1^\alpha$.
The following theorem expresses the fact that a subwavelength bandgap opens in the band structure of the bubbly crystal.

\begin{thm}{\rm{\cite{bandgap}}} \label{main_bandgap}
	For every $\epsilon>0$, there exists $\delta_0>0$  and $\tilde \omega > \omega_1^*$ such that 
	\begin{equation*}
	[ \omega_1^*+\epsilon, \tilde\omega ] \subset [\max_\alpha \omega_1^\alpha, \min_\alpha \omega_2^\alpha]
	\end{equation*} 
	for $\delta<\delta_0$.
\end{thm}

\section{Integral representation for bubbly crystals with a defect} \label{sec-2}

\subsection{Formulation of the line defect problem}

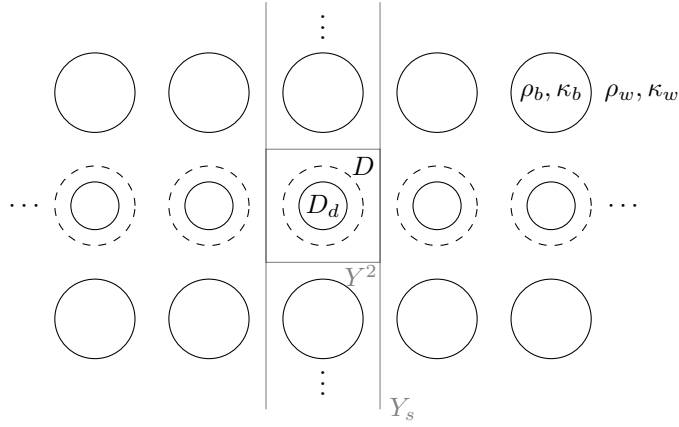
\begin{figure}[tb]
    \centering
    \begin{tikzpicture}[scale=1.5]
    
    \draw[dashed] (0,0) circle (10pt) node[yshift=15pt, xshift=15pt ]{$D$};
    \draw (0,0) circle (6pt) node{$D_d$};
    \draw[dashed] (-2,0) circle (10pt);
    \draw[dashed] (-1,0) circle (10pt);
    \draw[dashed] (2,0) circle (10pt);
    \draw[dashed] (1,0) circle (10pt);
    \draw (-2,0) circle (6pt);
	\draw (-1,0) circle (6pt);
	\draw (2,0) circle (6pt);
	\draw (1,0) circle (6pt);    
    
    \draw (0,1) circle (10pt);
    \draw (1,1) circle (10pt);
    \draw (0,-1) circle (10pt);
    \draw (1,-1) circle (10pt);
    \draw (-1,1) circle (10pt);
    \draw (-1,-1) circle (10pt);
    \draw (2,1) circle (10pt);
    \draw (2,-1) circle (10pt);
    \draw (-2,1) circle (10pt);
    \draw (-2,-1) circle (10pt);
    \draw (2.65,0) node{$\cdots$};
    \draw (-2.6,0) node{$\cdots$};
    \draw (0,1.65) node{$\vdots$};
    \draw (0,-1.5) node{$\vdots$};
    \draw (2.8,1) node{$\rho_w, \kappa_w$};
    \draw (2,1) node{$\rho_b, \kappa_b$};
    
    \draw[opacity = 0.5] (0.5,0.5) -- (0.5,-0.5) node[yshift=-5pt, xshift=-7pt ]{$Y^2$} -- (-0.5,-0.5) -- (-0.5,0.5) -- cycle;
    \draw[opacity = 0.5] (-0.5,-1.8) -- (-0.5,1.8) 
    					 (0.5,-1.8) node[right]{$Y_s$} -- (0.5,1.8);
    \end{tikzpicture}
    \caption{Illustration of the defect crystal and the material parameters.} \label{fig:defect}
\end{figure}
In the following, we will consider the case when all the bubbles are circular disks. This gives a convenient presentation, and makes the problem similar to the point defect problem studied in \cite{defectSIAM,defectX}. In Appendix \ref{app:noncirc}, we will outline the analysis in the case of non-circular bubbles.

Consider a perturbed crystal, where all the disks along the $x_1$-axis are replaced by defect disks of radius $R_d$ with $0< R_d<R$. Denote the centre defect disk by $D_d$ and let
$$\C_d = \left( \bigcup_{m\in \Z} D_d + (m,0)\right) \cup \left( \bigcup_{\substack{ m\in \Z \\ n\in\Z\setminus\{0\}}} D+(m,n) \right)$$
be the perturbed crystal, depicted in Figure \ref{fig:defect}. Moreover, let $\epsilon = R_d-R<0, \epsilon \in (-R,0)$ be the perturbation of the radius. Define
$$\rho_d(x) = \rho_b\chi_{\C_d}(x) + \rho_w(1-\chi_{\C}(x)), \quad \kappa_d(x) = \kappa_b\chi_{\C}(x) + \kappa_w(1-\chi_{\C_d}(x)).$$

The wave propagation problem inside the periodic crystal can be modelled as
\begin{equation}\label{eq:scattering} \kappa_d(x) \nabla \cdot \left(\frac{1}{\rho_d(x)} \nabla u(x) \right) + \omega^2 u(x) = 0, \quad x\in\R^2.\end{equation}
We  denote by $\Lambda_d$ the set of propagating frequencies in the line defect crystal, \ie{} the set of $\omega$ such that $\omega^2$ is in the spectrum of the operator
$$-\kappa_d \nabla \cdot \frac{1}{\rho_d} \nabla.$$

Since the defect crystal is periodic in the $x_1$-direction, we can use the Floquet transformation to decompose \eqnref{eq:scattering} as
\begin{equation}\label{eq:scattering-strip} \begin{cases}
\ds \kappa_d(x) \nabla \cdot \left(\frac{1}{\rho_d(x)} \nabla u(x) \right) + \omega^2 u(x) = 0, \quad x \in Y_s, \\ 
\ds e^{-i \alpha_1 x_1} u  \,\,\,  \mbox{is periodic in } x_1,
\end{cases}
\end{equation}
where $\alpha_1 \in Y^*$ and $Y_s$ again denotes the strip $Y_s = [-1/2,1/2)\times \R$. We will denote by $\Lambda_{d,\alpha_1}$ the set of $\omega$ such that $\omega^2$ is in the spectrum of the operator implied by \eqnref{eq:scattering-strip} and by $\Lambda^{ess}_{d,\alpha_1}$ the corresponding essential part of the spectrum. 

In the strip $Y_s$, the perturbations $\rho_d-\rho$ and $\kappa_d-\kappa$ have compact support. Since the essential spectrum is stable under compact perturbations \cite{Figotin,MMMP4}, it can be shown that the essential spectra $\Lambda^{ess}_{0,\alpha_1}$ and $\Lambda^{ess}_{d,\alpha_1}$ coincide.

In this paper, we want to show that introducing the line defect creates a defect band $\omega^\epsilon(\alpha_1) \notin \Lambda_{0,\alpha_1}$. Moreover, we want to show that $\epsilon$ can be chosen such that $\omega^\epsilon(\alpha_1) \notin \Lambda_0$ for all $\alpha_1 \in Y^*$, which means that any Bloch mode is localized to the line defect. We also want to show that $\omega^\epsilon(\alpha_1)$ is not contained in the pure point part of $\Lambda_{d,\alpha_1}$, which means that there are no bound modes in the defect direction.

\subsection{Effective sources for the defect}
Here we describe an effective sources approach to
the solution of \eqnref{eq:scattering-strip} in the strip. The idea is to model the defect bubble $D_d$ as an unperturbed bubble $D$ with additional fictitious monopole and dipole sources $f$ and $g$. This method was originally introduced in \cite{Wilcox} and then it was applied in \cite{defectSIAM,defectX} for a point defect in a bubbly crystal.

Let us consider the following problem:
\begin{equation} \label{eq:scattering_fictitious}
\left\{
\begin{array} {ll}
	&\ds \nabla \cdot \frac{1}{\rho_w} \nabla  \widetilde{u}+ \frac{\omega^2}{\kappa_w} \widetilde{u}  = 0 \quad \text{in} \quad Y_s \setminus \C, \\
	\nm
	&\ds \nabla \cdot \frac{1}{\rho_b} \nabla  \widetilde{u}+ \frac{\omega^2}{\kappa_b} \widetilde{u}  = 0 \quad \text{in} \quad Y_s\cap\C, \\
	\nm
	&\ds  \widetilde{u}|_{+} -\widetilde{u}|_{-}  =f\delta_{m,0}  \quad \text{on} \quad \partial D+(0,m), \ m\in\Z, \\
	\nm
	& \ds  \frac{1}{\rho_w} \frac{\partial \widetilde{u}}{\partial \nu} \bigg|_{+} - \frac{1}{\rho_b} \frac{\partial \widetilde{u}}{\partial \nu} \bigg|_{-} =g\delta_{m,0} \quad \text{on} \quad \partial D+(0,m), \ m\in\Z, \\
	\nm
	& \ds e^{-i \alpha_1 x_1} \widetilde u  \,\,\,  \mbox{is periodic in } x_1,
\end{array}
\right.
\end{equation} 
where $f$ and $g$ are the source terms and $\delta_{m,n}$ is the Kronecker delta function.
Note that the sources are present only on the boundary of the central bubble $D$. 

We  denote the solution to the original problem \eqnref{eq:scattering-strip} by $u$ and the effective source solution \eqref{eq:scattering_fictitious} by $\widetilde{u}$. We want to find appropriate conditions on $f$ and $g$ in order to achieve
\begin{equation}\label{eq:coincide}
u \equiv \widetilde{u}  \quad \mbox{in } (Y_s\setminus D) \cup D_d.
\end{equation}
Then $u$ can be recovered by extending $\widetilde{u}$ to the whole region including $D\setminus D_d$ with boundary conditions on $\p D$ and $\p D_d$.
The conditions for the effective sources $f$ and $g$, which are necessary in order to correctly model the defect, will be characterized in the next subsection.

\subsection{Characterization of the effective sources}\label{subsec:char_eff}
Here we clarify the relation between the effective source pair $(f,g)$ and the layer density pair $(\varphi,\psi)$ defined in equation \eqnref{eq:u_rep_effective} below. 

First, we observe that away from the central unit cell $Y^2$, the equations \eqnref{eq:scattering-strip} and \eqref{eq:scattering_fictitious} satisfy the same geometric and quasi-periodic conditions. Thus, in order for \eqnref{eq:coincide} to hold, it is sufficient for $u$ and  $\widetilde{u}$ to coincide inside the central unit cell $Y^2$.

Inside $Y^2$, the solution $\widetilde{u}$ can be represented as
\begin{align}\label{eq:u_rep_effective}
\widetilde{u}
= 
\begin{cases}
H + \S_D^{k_w}[\psi] &\quad \mbox{in } Y^2\setminus \overline{D},
\\
\S_D^{k_b}[\varphi] &\quad \mbox{in } D,
\end{cases}
\end{align}
for some pair $(\varphi,\psi)\in L^2(\p D)^2$, with $H$ satisfying the homogeneous equation $(\Delta + k_w^2) H = 0$ in $Y^2$. In (\ref{eq:u_rep_effective}), the local properties of $\widetilde{u}$ around $\partial D$ are given by the single-layer potentials, while $H$ can be chosen to make $\widetilde{u}$ satisfy the quasi-periodic condition. From the jump conditions given in Section \ref{sec:layerpot}, the pair $(\varphi,\psi)$ satisfies
\begin{equation}\label{eq:AD}
\mathcal{A}_D
\begin{pmatrix}
\varphi 
\\
\psi
\end{pmatrix}
:=
\begin{pmatrix}
\S_D^{k_b} & - \S_D^{k_w}
\\[0.3em]
\ds \partial \S_D^{k_b}/\partial \nu |_{-}
&
\ds -\delta \partial \S_D^{k_w} / \partial \nu|_{+}
\end{pmatrix}
\begin{pmatrix}
\varphi 
\\
\psi
\end{pmatrix}
=
\begin{pmatrix}
H|_{\partial D}- f
\\[0.3em]
\ds\partial H/\partial \nu  |_{\partial D} - g
\end{pmatrix}.
\end{equation}

Similarly, inside $Y^2$, the solution $u$ can be represented as
\begin{align*}
u=
\begin{cases}
H + \S_{D_d}^{k_w}[\psi_d] &\quad \mbox{in } Y^2\setminus \overline{D_d},
\\
\S_{D_d}^{k_b}[\varphi_d] &\quad \mbox{in } D_d,
\end{cases}
\end{align*}
where
\begin{equation}\label{eq:ADd_original}
\A_{D_d}
\begin{pmatrix}
\varphi_d 
\\
\psi_d
\end{pmatrix}:=
\begin{pmatrix}
\S_{D_d}^{k_b} & - \S_{D_d}^{k_w}
\\[0.3em]
\ds \partial \S_{D_d}^{k_b} /\partial \nu|_{-}
&
\ds -\delta \partial \S_{D_d}^{k_w}/\partial \nu  |_{+}
\end{pmatrix}
\begin{pmatrix}
\varphi_d 
\\
\psi_d
\end{pmatrix}
=
\begin{pmatrix}
H|_{\partial {D_d}}
\\[0.3em]
\ds\partial H /\partial \nu |_{\partial {D_d}}
\end{pmatrix}.
\end{equation}

Now, having the two solutions coincide inside $(Y^2\setminus D) \cup D_d$ is equivalent to the conditions
\begin{equation}\label{eq:SDdSD_inside}
\mathcal{S}_{D_d}^{k_b}[\varphi_d] \equiv \mathcal{S}_D^{k_b}[\varphi] \quad\mbox{in }D_d,
\end{equation}
and
\begin{equation}\label{eq:SDdSD_outside}
\mathcal{S}_{D_d}^{k_w}[\psi_d] \equiv \mathcal{S}_D^{k_w}[\psi] \quad\mbox{in }Y^2\setminus \overline{D}.
\end{equation}
Assuming $D$ is a disk, the above equations were solved in \cite{defectSIAM,defectX}, and we state the results in Proposition \ref{prop:effective} below. First, we introduce some notation. 
Since $D$ and $D_d$ are circular disks, we can use a Fourier basis for functions in $L^2(\p D)$ or $L^2(\p D_d)$. For $n\in \Z$, define the subspace $V_n$ of $L^2(\p D)$ as $V_n : = \mbox{span}\{ e^{im \theta}\}$. Then define the subspace $V_{mn}$ of $L^2(\p D)^2$ as 
$$
V_{mn} : = V_m \times V_n, \quad
m,n\in \mathbb{Z}.
$$
Similarly, let ${V}_{mn}^d$ be the subspace of $L^2(\p D_d)^2$ with the same Fourier basis. Then it can be shown that the operator $\A_{D}$ in \eqref{eq:AD} has the following matrix representation as an operator from $V_{mn}$ to $V_{m'n'}$:
\begin{equation*}\label{eq:AD_multipole}
(\A_{D})_{V_{mn}\rightarrow V_{m'n'}} = \delta_{mn}\delta_{m'n'} \frac{(-i)\pi R}{2} \begin{pmatrix}
  J_n(k_b R)H_n^{(1)}(k_b R) & -  J_n(k_w R)H_n^{(1)}(k_w R)
 \\
 k_b J_n'(k_b R) H_n^{(1)}(k_b R) & - \delta k_w J_n(k_w R) \big(H_n^{(1)}\big)'(k_w R)
 \end{pmatrix}.
\end{equation*}
Similarly, the operator $\A_{D_d}$ in \eqref{eq:ADd_original} is represented as follows:
\begin{equation*}\label{eq:ADd_multipole}
(\A_{D_d})_{V_{mn}^d\rightarrow V_{m'n'}^d} = \delta_{mn}\delta_{m'n'} \frac{(-i)\pi R_d}{2} \begin{pmatrix}
  J_n(k_b R_d)H_n^{(1)}(k_b R_d) & -  J_n(k_w R_d)H_n^{(1)}(k_w R_d)
 \\
 k_b J_n'(k_b R_d) H_n^{(1)}(k_b R_d) & - \delta k_w J_n(k_w R_d) \big(H_n^{(1)}\big)'(k_w R_d)
 \end{pmatrix}.
\end{equation*}
In \cite{defectSIAM,defectX}, the following proposition was shown.
\begin{prop} \label{prop:effective}
	The density pair $(\varphi,\psi)$ and the effective sources $(f,g)$ satisfy the following relation 
	\begin{equation*}\label{eq:relation_density_source}
	(\mathcal{A}_D^\epsilon - \mathcal{A}_D)\begin{pmatrix}
	\varphi
	\\[0.3em]
	\psi
	\end{pmatrix} = 
	\begin{pmatrix}
	f
	\\[0.3em]
	g
	\end{pmatrix},
	\end{equation*}
	where the operators $\mathcal{P}_1: L^2(\p D)^2\rightarrow L^2(\p D_d)^2$ and $\mathcal{P}_2: L^2(\p D)^2\rightarrow L^2(\p D_d)^2$ are defined by
	\begin{align*}
	(\mathcal{P}_1)_{V_{mn}\rightarrow V_{m'n'}^d} &= \delta_{mn}   \delta_{m'n'}
	\frac{R}{R_d}\begin{pmatrix}
	\ds\frac{H_n^{(1)}(k_b R) }{H_n^{(1)}(k_b R_d)}
	& 0
	\\
	0& \ds\frac{J_n(k_w R) }{J_n(k_w R_d)}
	\end{pmatrix},
	\\
	(\mathcal{P}_2)_{V_{mn}\rightarrow V_{m'n'}^d} &= \delta_{mn}   \delta_{m'n'}
	\begin{pmatrix}
	\ds \frac{J_n(k_w R_d) }{J_n(k_w R)}
	& 0
	\\
	0& \ds\frac{J_n'(k_w R_d) }{J_n'(k_w R)}
	\end{pmatrix},
	\end{align*}
	and $\mathcal{A}_D^\epsilon$ is defined as
	\begin{equation}\label{eq:ADd}
	\mathcal{A}_D^\epsilon  := (\mathcal{P}_2)^{-1}\mathcal{A}_{D_d} \mathcal{P}_1 .
	\end{equation}
\end{prop}

\subsection{Floquet transform of the solution}
In view  of Proposition \ref{prop:effective}, we can identify the solutions $u$ and $\widetilde{u}$. In this section, we derive an integral equation for the effective source problem \eqref{eq:scattering_fictitious}. This problem is already quasi-periodically reduced in the $x_1$-direction, with quasi-periodicity $\alpha_1$. For some quasi-periodicity $\alpha_2 \in Y^*$, we set $\alpha = (\alpha_1,\alpha_2)$ and apply the Floquet transform to the solution $u$ in the $x_2$-direction as follows:
$$
u^\alpha = \sum_{m\in \Z} u(x-(0,m)) e^{i\alpha_2 m}.
$$
The transformed solution $u^\alpha$ satisfies
\begin{equation*} \label{eq:scattering_quasi}
\left\{
\begin{array} {ll}
	&\ds \nabla \cdot \frac{1}{\rho_w} \nabla  u^\alpha+ \frac{\omega^2}{\kappa_w} u^\alpha  = 0 \quad \text{in} \quad Y^2 \setminus \overline{D}, \\
	\nm
	&\ds \nabla \cdot \frac{1}{\rho_b} \nabla  u^\alpha+ \frac{\omega^2}{\kappa_b} u^\alpha  = 0 \quad \text{in} \quad D, \\
	\nm
	&\ds  u^\alpha|_{+} -u^\alpha|_{-}  =f   \quad \text{on} \quad \partial D, \\
	\nm
	& \ds  \frac{1}{\rho_w} \frac{\partial u^\alpha}{\partial \nu} \bigg|_{+} - \frac{1}{\rho_b} \frac{\partial u^\alpha}{\partial \nu} \bigg|_{-} =g \quad \text{on} \quad \partial D,
	\\
	\nm
	& \ds e^{-i \alpha \cdot x} u^\alpha \text{ is periodic}.
\end{array}
\right.
\end{equation*} 
The solution $u^\alpha$ is $\alpha$-quasi-periodic in the two-dimensional cell $Y^2$, and can be represented using quasi-periodic layer potentials as
\begin{align*}
u^\alpha
= 
\begin{cases}
\S_D^{\alpha,k_w}[\psi_\alpha], &\quad \mbox{in } Y^2\setminus \overline{D},
\\
\S_D^{k_b}[\varphi_\alpha], &\quad \mbox{in } D,
\end{cases}
\end{align*}
where, similarly as in equation \eqnref{eq-boundary}, the pair $(\varphi^\alpha,\psi^\alpha)\in L^2(\p D)^2$ is the solution to
\begin{equation*}\label{eq:phipsialpha}
\mathcal{A}^\alpha(\omega,\delta)
\begin{pmatrix}
\varphi^\alpha 
\\[0.3em]
\psi^\alpha
\end{pmatrix}
:=  \begin{pmatrix}
\mathcal{S}_D^{k_b} &  -\mathcal{S}_D^{\alpha,k}  \\[0.3em]
-\frac{1}{2}+ \K_D^{k_b,*}& -\delta\left( \frac{1}{2}+ \left(\mathcal{K}_D^{ -\alpha,k}\right)^*\right)
\end{pmatrix}
\begin{pmatrix}
\varphi^\alpha 
\\[0.3em]
\psi^\alpha
\end{pmatrix}
=
\begin{pmatrix}
- f
\\[0.3em]
- g
\end{pmatrix}.
\end{equation*}
Since the operator $\mathcal{A}^\alpha$ is invertible for small enough $\delta$ and for $\omega$ inside the bandgap  \cite{bandgap}, we have
$$
\begin{pmatrix}
\varphi^\alpha 
\\[0.3em]
\psi^\alpha
\end{pmatrix} 
= \mathcal{A}^\alpha(\omega,\delta)^{-1}\begin{pmatrix}
- f
\\[0.3em]
- g
\end{pmatrix}.
$$
The solution $u$ to problem \eqnref{eq:scattering_fictitious} can be recovered by the inversion formula as
\begin{equation*}
u(x)=\frac{1}{2\pi}\int_{Y^*} u^{(\alpha_1,\alpha_2)}(x) \dx\alpha_2.
\end{equation*}
Now, by the same arguments as those in \cite{defectSIAM,defectX}, we obtain the following proposition.
\begin{prop} \label{prop:floquet}
	The density pair $(\varphi,\psi)$ and the effective source pair $(f,g)$ satisfy
	\begin{equation}\label{eq:IntAa}
	\begin{pmatrix}
	\varphi
	\\[0.3em]
	\psi 
	\end{pmatrix}
	=\left(\frac{1}{2\pi}\int_{Y^*} \mathcal{A}^{(\alpha_1,\alpha_2)}(\omega,\delta)^{-1} \dx\alpha_2 \right)\begin{pmatrix}
	-f
	\\[0.3em]
	-g 
	\end{pmatrix},
	\end{equation}
	for small enough $\delta$ and for $\omega \notin \Lambda_{0,\alpha_1}$ inside the bandgap.
\end{prop}

\subsection{The integral equation for the layer densities}
Here we state the integral equation for the layer density pair $(\phi,\psi)$. The following result is an immediate consequence of Propositions  \ref{prop:effective} and  \ref{prop:floquet}.
\begin{prop} \label{prop:fg}
	The layer density pair $(\phi,\psi)\in L^2(\partial D)^2$ satisfies the following integral equation:
	\begin{align} \label{eq:Mdensity}
	\mathcal{M}^{\epsilon,\delta,\alpha_1}(\omega)\begin{pmatrix}
	\phi
	\\
	\psi
	\end{pmatrix}:=\bigg(I+\left(\frac{1}{2\pi}\int_{Y^*}\A^\alpha(\omega,\delta)^{-1} \dx\alpha_2\right)(\A_{D}^\epsilon(\omega,\delta) - \A_{D}(\omega,\delta))\bigg)
	\begin{pmatrix}
	\phi
	\\[0.3em]
	\psi
	\end{pmatrix}
	=
	\begin{pmatrix}
	0
	\\[0.3em]
	0
	\end{pmatrix},
	\end{align}
	for small enough $\delta$ and for $\omega \notin \Lambda_{0,\alpha_1}$ inside the bandgap.
\end{prop}
The expression of this integral equation resembles the one for a point defect found in \cite{defectSIAM,defectX}. However, this similarity is not obvious, and can be seen as a consequence of the cancellation of $H$ in Proposition \ref{prop:effective}.

The significance of Proposition \ref{prop:fg} is as follows. If we can show that there is a characteristic value $\omega = \omega^\epsilon$ of $\mathcal{M}^{\epsilon,\delta,\alpha_1}$ inside the bandgap, \ie{} if there is a non-trivial pair $(\phi,\psi)$ such that $\mathcal{M}^{\epsilon,\delta,\alpha_1}(\omega^\epsilon) \left( \begin{smallmatrix}
\phi \\ \psi\end{smallmatrix} \right) = 0$, then $\omega^\epsilon$ is a resonance frequency for the defect mode. 

\section{Subwavelength guided modes in the defect} \label{sec-3}
Here, we will prove the existence of a resonance frequency $\omega = \omega^\epsilon(\alpha_1)$ inside the bandgap of the unperturbed crystal at $\alpha_1$. We will give an asymptotic formula for $\mathcal{M}^{\epsilon,\delta,\alpha_1}$ in terms of $\delta$ in the dilute regime. Moreover, we will show that the defect band is not contained in the pure point spectrum of the defect operator, and for perturbation sizes $\epsilon$ larger than some critical $\epsilon_0$, the entire defect band is located in the bandgap region of the original operator.

\subsection{Asymptotic expansions for small $\delta$}
In this section, we will asymptotically expand $\mathcal{M}^{\epsilon,\delta,\alpha_1}$ in the limit as $\delta \rightarrow 0$ and with $\omega$ in the subwavelength regime, \ie{}, $\omega = O(\sqrt{\delta})$. Throughout this section we assume that $\alpha \neq (0,0)$. We  begin by studying the operator $(\A^\alpha(\omega, \delta))^{-1}$.

Define $\psi_{\alpha}$ as 
$$\psi_{\alpha} = \left(\S_D^{\alpha,0}\right)^{-1} [\chi_{\p D}].$$
Since we know that $\big(\frac{1}{2}I + (\K_D^{-\alpha,0})^*\big)[\psi_\alpha] = \psi_\alpha$, we can decompose this operator as
$$ \frac{1}{2}I + (\K_D^{-\alpha,0})^* = P_\alpha +  Q_\alpha,$$
where
$$P_\alpha = -\frac{\langle\chi_{\p D}, \cdot\rangle}{\capacity_{D,\alpha}}\psi_\alpha$$
is a projection on $\psi_\alpha$. Then it can be shown that $Q_\alpha[\psi_\alpha] = 0$ and $Q_\alpha^* [\chi_{\p D}] = 0$, where $Q_\alpha^*$ is the adjoint of $Q_\alpha$.

For small $\delta$ and for $\omega=O(\sqrt{\delta})$ inside the corresponding bandgap, the operator $\A^\alpha(\omega, \delta)$ can be decomposed as 
$$\A^\alpha(\omega, \delta) = \begin{pmatrix}
\mathcal{S}_D^{\omega} &  -\mathcal{S}_D^{\alpha,\omega}  \\
-\frac{1}{2}I+ \mathcal{K}_D^{\omega, *}& 0 
\end{pmatrix} - \delta \begin{pmatrix}
0 & 0  \\
0 & P_\alpha
\end{pmatrix} - \delta \begin{pmatrix}
0 & 0  \\
0 & Q_\alpha
\end{pmatrix} + O(\delta^3).$$
Define the operators
$$A_0 = \begin{pmatrix}
\mathcal{S}_D^{\omega} &  -\mathcal{S}_D^{\alpha,\omega}  \\
-\frac{1}{2}I+ \mathcal{K}_D^{\omega, *}& 0 
\end{pmatrix} , $$
and
$$A_1 = I - \delta A_0^{-1} \begin{pmatrix}
0 & 0  \\
0 & P_\alpha
\end{pmatrix}.$$
The motivation for defining these operators is given in Lemmas \ref{lem:invA0A1} and \ref{lem:invAa}. Introducing these operators enables the explicit computation of $(A^\alpha)^{-1}$. We will compute the asymptotic expansion of these operators for small $\omega$ and $\delta$.

\begin{lem} \label{lem:invA0A1}
	The following results hold for $A_0$ and $A_1$:
	\begin{itemize}
		\item[(i)] For $\omega \neq 0$, $A_0: L^2(\p D) \rightarrow L^2(\p D)$ is invertible, and as $\omega \rightarrow 0$ and $\delta \rightarrow 0$,
$$A_0^{-1} = \begin{pmatrix}
0 &  -\frac{\langle \chi_{\p D}, \cdot \rangle}{\pi R^3\omega^2\ln \omega} \chi_{\p D} +O\left(\frac{1}{\omega\ln\omega}\right) \\
-\left(\mathcal{S}_D^{\alpha,0}\right)^{-1} + O(\omega^2) & -\frac{\langle \chi_{\p D}, \cdot \rangle}{\pi R^2\omega^2}\psi_\alpha +O\left(\frac{1}{\omega}\right)
\end{pmatrix}.$$
		
		\item[(ii)]  For $\omega \neq \omega^\alpha$, $A_1: L^2(\p D) \rightarrow L^2(\p D)$ is invertible, and as $\omega \rightarrow 0$ and $\delta \rightarrow 0$,
$$A_1^{-1} = \begin{pmatrix}
I & -\frac{(\omega^\alpha)^2}{\omega^2R\ln \omega} \frac{\left\langle \chi_{\p D}, \left(P_\alpha^\perp\right)^{-1}[\cdot]  \right\rangle}{\capacity_{D,\alpha}}\chi_{\p D} + O\left(\frac{\omega}{\ln\omega}\right)\\
0 & \left(P_\alpha^\perp\right)^{-1} +O(\omega)
\end{pmatrix},$$
where $P_\alpha^\perp = I-\frac{(\omega^\alpha)^2}{\omega^2}P_\alpha$.
	\end{itemize}
\end{lem}
\noindent \textit{Proof of (i).}
We easily find that 
\begin{equation}\label{eq:Alem1}A_0^{-1} = \begin{pmatrix}
0 &  \left(-\frac{1}{2}I+ \mathcal{K}_D^{\omega, *}\right)^{-1} \\
-\left(\mathcal{S}_D^{\alpha,\omega}\right)^{-1} & \left(\mathcal{S}_D^{\alpha,\omega}\right)^{-1} \mathcal{S}_D^{\omega} \left(-\frac{1}{2}I+ \mathcal{K}_D^{\omega, *}\right)^{-1}
\end{pmatrix},\end{equation}
which is well-defined since $-\frac{1}{2}I+ \mathcal{K}_D^{\omega, *}: L^2(\p D) \rightarrow L^2(\p D)$ is invertible for $\omega\neq0$ \cite{MaCMiPaP}. From the low-frequency expansion of $\mathcal{S}_D^{\alpha,\omega}$ \cite{MaCMiPaP}, and using the Neumann series, we have 
\begin{align} \label{eq:Alem2}
\left(\mathcal{S}_D^{\alpha,\omega}\right)^{-1} &= \left(\mathcal{S}_D^{\alpha,0}+O(\omega^2)\right)^{-1} \nonumber \\
&=\left(\mathcal{S}_D^{\alpha,0}\right)^{-1} +O(\omega^2).
\end{align}
Using the Fourier basis, the operator $-\frac{1}{2}I+ \mathcal{K}_D^{\omega,*}$ can be represented as \cite{bandgap}
$$\left(-\frac{1}{2}I+ \mathcal{K}_D^{\omega, *}\right)_{V_{m}\rightarrow V_{n}} = \delta_{mn}\left(-\frac{1}{2} + \frac{-i\pi R \omega}{4}\left(H_n^{(1)}(\omega R)J_n'(\omega R) + (H_n^{(1)})'(\omega R)J_n(\omega R)\right) \right).
$$
Using standard asymptotics we can compute
\begin{align*}
\left(-\frac{1}{2}I+ \mathcal{K}_D^{\omega, *}\right)_{V_{n}\rightarrow V_{n}} = \begin{cases} -\frac{R^2}{2}\omega^2\left(2\pi\eta_\omega + \ln R\right) + O(\omega^3\ln\omega) \quad &n=0 ,\\ -\frac{1}{2} + O(\omega) & n\neq0  .\end{cases} 
\end{align*}
Hence the operator $\left(-\frac{1}{2}I+ \mathcal{K}_D^{\omega, *}\right)^{-1}$ can be written as
\begin{equation}\label{eq:Alem3}\left(-\frac{1}{2}I+ \mathcal{K}_D^{\omega, *}\right)^{-1} = -\frac{1}{\pi R^3\omega^2\left(2\pi\eta_\omega + \ln R\right)}\langle \chi_{\p D}, \cdot \rangle \chi_{\p D} + O\left(\frac{1}{\omega\ln\omega}\right).\end{equation}
Moreover, we have from \eqnref{eq:Sexpansion} that $\mathcal{S}_D^{\omega}[\chi_{\p D}] = (2\pi R\eta_\omega+ R\ln R)\chi_{\p D} + O(\omega^2\ln\omega)$, and so 
\begin{equation}\label{eq:Alem4}\left(\mathcal{S}_D^{\alpha,\omega}\right)^{-1} \mathcal{S}_D^{\omega} \left(-\frac{1}{2}I+ \mathcal{K}_D^{\omega, *}\right)^{-1} = -\frac{\langle \chi_{\p D}, \cdot \rangle}{\pi R^2\omega^2}\psi_\alpha +O\left(\frac{1}{\omega}\right).\end{equation}
Combining equations \eqnref{eq:Alem1}, \eqnref{eq:Alem2}, \eqnref{eq:Alem3} and \eqnref{eq:Alem4} proves (i).
\qed
\bigskip

\noindent \textit{Proof of (ii).}
Using the definition of $A_1$, and the expression for $A_0$, we can compute
$$A_1 = I - \delta \begin{pmatrix}
0 &  -\frac{\langle \chi_{\p D}, \cdot \rangle}{\pi R^3\omega^2\ln \omega} \chi_{\p D} +O\left(\frac{1}{\omega\ln\omega}\right) \\
0 & -\frac{\langle \chi_{\p D}, \cdot \rangle}{\pi R^2\omega^2}\psi_\alpha +O\left(\frac{1}{\omega}\right)
\end{pmatrix}.$$
Recall the asymptotic expression of $\omega^\alpha$ given in Theorem \ref{approx_thm}:
\begin{equation}\label{eq:delta}
\omega^\alpha = \sqrt{\frac{\delta\capacity_{D,\alpha}}{\pi R^2}} +O(\delta^{3/2}).
\end{equation}
We then find that
$$A_1 = \begin{pmatrix}
I &  \frac{(\omega^\alpha)^2}{\omega^2R\ln \omega} \frac{\langle \chi_{\p D}, \cdot \rangle}{\capacity_{D,\alpha}}\chi_{\p D} +O\left(\frac{\omega}{\ln\omega}\right) \\
0 & I-\frac{(\omega^\alpha)^2}{\omega^2}P_\alpha +O(\omega)
\end{pmatrix}.$$
Define $P_\alpha^\perp = I-\frac{(\omega^\alpha)^2}{\omega^2}P_\alpha$. For $\omega$ small enough, $A_1$ is invertible precisely when $P_\alpha^\perp$ is invertible, \ie{} when $\omega\neq \omega^\alpha$. Moreover, we have
$$A_1^{-1} = \begin{pmatrix}
I & -\frac{(\omega^\alpha)^2}{\omega^2R\ln \omega} \frac{\left\langle \chi_{\p D}, \left(P_\alpha^\perp\right)^{-1}[\cdot]  \right\rangle}{\capacity_{D,\alpha}}\chi_{\p D} + O\left(\frac{\omega}{\ln\omega}\right)\\
0 & \left(P_\alpha^\perp\right)^{-1} +O(\omega)
\end{pmatrix}.$$
This proves (ii).
\qed
\begin{lem} \label{lem:invAa}
	For $\omega \neq \omega^\alpha$, and as $\omega \rightarrow 0$ and $\delta \rightarrow 0$, we have
	\begin{align*}
	(\A^\alpha(\omega, \delta))^{-1} &= A_1^{-1}A_0^{-1}\big(I + O(\delta)\big).
	\end{align*}
\end{lem}
\begin{proof}
	We have already established the invertibility of $A_0$ and $A_1$. Using this fact, we have
	\begin{align*}
	\A^\alpha(\omega, \delta) &= A_0 - \delta \begin{pmatrix}
	0 & 0  \\
	0 & P_\alpha
	\end{pmatrix} - \delta \begin{pmatrix}
	0 & 0  \\
	0 & Q_\alpha
	\end{pmatrix} + O(\delta^3) \\
	&= A_0\left( I - \delta A_0^{-1}\begin{pmatrix}
	0 & 0  \\
	0 & P_\alpha
	\end{pmatrix} - \delta A_0^{-1}\begin{pmatrix}
	0 & 0  \\
	0 & Q_\alpha
	\end{pmatrix} + O(\delta^2) \right)  \\ 
	&= A_0A_1\left( I - \delta A_1^{-1}A_0^{-1}\begin{pmatrix}
	0 & 0  \\
	0 & Q_\alpha
	\end{pmatrix}+ O(\delta^2)\right).
	\end{align*}
	Because $Q_\alpha^* \chi_{\p D} = 0$, we have that
	$$\delta A_1^{-1}A_0^{-1}\begin{pmatrix}
	0 & 0  \\
	0 & Q_\alpha
	\end{pmatrix} = O(\delta).$$
	We then have that
	\begin{align*}
	(\A^\alpha(\omega, \delta))^{-1} &= A_1^{-1}A_0^{-1}\big(I + O(\delta)\big)^{-1} \\ &= A_1^{-1}A_0^{-1}\big(I + O(\delta)\big),
	\end{align*}
	where the last step follows using the Neumann series.
\end{proof}
Next, we compute the operator $(\mathcal{A}_D^\epsilon - \mathcal{A}_D)$. Using Proposition \ref{prop:effective} and equation (\ref{eq:ADd}), we have
\begin{equation*}
(\A_{D}^\epsilon)_{V_{mn}\rightarrow V_{m'n'}} = \delta_{mn}\delta_{m'n'} \frac{(-i)\pi R}{2} \begin{pmatrix} J_n(\omega R)H_n^{(1)}(\omega R) & -  J_n(\omega R)H_n^{(1)}(\omega R_d) \frac{J_n(\omega R) }{J_n(\omega R_d)}\\\omega J_n'(\omega R) H_n^{(1)}(\omega R) & - \delta \omega J_n(\omega R) \big(H_n^{(1)}\big)'(\omega R_d) \frac{J_n'(\omega R)}{J_n'(\omega R_d) }\end{pmatrix}.
\end{equation*}	
Consequently, the operator $(\mathcal{A}_D^\epsilon - \mathcal{A}_D)$ is given by
\begin{align*}
(\mathcal{A}_D^\epsilon - \mathcal{A}_D)_{V_{mn}\rightarrow V_{m'n'}} =& \delta_{mn}\delta_{m'n'} \frac{(-i)\pi RJ_n(\omega R)}{2} \begin{pmatrix} 0 & H_n^{(1)}(\omega R) - \frac{J_n(\omega R)H_n^{(1)}(\omega R_d)  }{J_n(\omega R_d)}\\0 & \delta \omega \left(\big(H_n^{(1)}\big)'(\omega R) - \frac{J_n'(\omega R)\big(H_n^{(1)}\big)'(\omega R_d)}{J_n'(\omega R_d)} \right)\end{pmatrix}.
\end{align*}	
Introduce the notation
\begin{equation}  \label{eq:Apert}
\mathcal{A}_D^\epsilon - \mathcal{A}_D := \begin{pmatrix} 0 & E_1^\epsilon \\0 & E_2^\epsilon\end{pmatrix}.
\end{equation}
Using asymptotic expansions of the Bessel function $J_n(z)$ and the Hankel function $H_n^{(1)}(z)$, for small $z$, straightforward computations show that
\begin{align*}
(E_1^\epsilon)_{V_{m}\rightarrow V_{n}} &= \delta_{m,n}\frac{(-i)\pi R}{2}\frac{J_n(\omega R)}{J_n(\omega R_d)}\left(H_n^{(1)}(\omega R){J_n(\omega R_d)} - J_n(\omega R)H_n^{(1)}(\omega R_d) \right), \\ &= \begin{cases} \ds \delta_{m,n}\left(R\ln\frac{R}{R_d} + O(\omega\ln\omega)\right), \qquad &n = 0 ,\\[0.5em] \ds \delta_{m,n} \left(-\frac{R}{2|n|}\left(1-\frac{R^{2|n|}}{R_d^{2|n|}}\right) + O(\omega)\right), &n\neq 0.\end{cases}
\end{align*}
Moreover, we have 
\begin{align} \label{eq:E2}
(E_2^\epsilon)_{V_{m}\rightarrow V_{n}} &= \delta_{m,n}\frac{(-i)\pi RJ_n(\omega R)}{2} \delta \omega \left(\big(H_n^{(1)}\big)'(\omega R) - \frac{J_n'(\omega R)\big(H_n^{(1)}\big)'(\omega R_d)}{J_n'(\omega R_d)} \right), \\ \nonumber &= \begin{cases} \ds \delta_{m,n}\left( \delta \left(1-\frac{R^2}{R_d^2} \right) + O(\delta\omega^2\ln\omega)\right), \qquad &n = 0, \\[0.5em]  \ds \delta_{m,n}O(\delta), &n\neq 0.\end{cases}
\end{align}

We are now ready to compute the full operator $\mathcal{M}^{\epsilon,\delta,\alpha_1}$. 
\begin{prop}
	The operator $\M^{\epsilon,\delta,\alpha_1}(\omega)$ has the form 
	\begin{equation} \label{eq:Mmatrix}
	\M^{\epsilon,\delta,\alpha_1}(\omega) = \begin{pmatrix}
	I &  M_1(\omega) \\
	0 & I + M_0(\omega)
	\end{pmatrix}, 
	\end{equation}
	where the operators $M_0(\omega), M_1(\omega): L^2(\p D) \rightarrow L^2(\p D)$ depend on $\epsilon,\delta,\alpha_1$. Moreover, as $\omega\rightarrow 0$, $\delta\rightarrow 0$ and $\omega \notin \Lambda_{0,\alpha_1}$ we have
	$$M_0(\omega) = -\frac{1}{2\pi}\int_{Y^*} \left( \left(P_\alpha^\perp\right)^{-1}(\S_{D}^{\alpha,0})^{-1}E_1^\epsilon + \delta \left(1-\frac{R^2}{R_d^2} \right)\frac{\langle \chi_{\p D}, \cdot \rangle}{\pi R^2\left(\omega^2-(\omega^\alpha)^2\right)}\psi_\alpha \right) \dx \alpha_2 +O(\omega).$$
\end{prop}

\begin{proof}
	The expression of $\mathcal{M}^{\epsilon,\delta,\alpha_1}$ given in  \eqnref{eq:Mmatrix} follows from equations \eqnref{eq:Mdensity} and \eqnref{eq:Apert}.
	Combining Lemmas \ref{lem:invA0A1} and  \ref{lem:invAa}, we find that
	$$(\A^\alpha(\omega, \delta))^{-1} = 
	\begin{pmatrix}
	\frac{(\omega^\alpha)^2}{\omega^2R\ln \omega} \frac{\left\langle \chi_{\p D}, \left(P_\alpha^\perp\right)^{-1}\left(\mathcal{S}_D^{\alpha,0}\right)^{-1}[\cdot] \right\rangle}{\capacity_{D,\alpha}}\chi_{\p D} + O\left(\frac{\omega}{\ln\omega}\right) &  -\frac{\langle \chi_{\p D}, \cdot \rangle}{\pi R^3\left(\omega^2-(\omega^\alpha)^2\right)\ln \omega} \chi_{\p D} +O\left(\frac{1}{\omega\ln\omega}\right) \\
	-\left(P_\alpha^\perp\right)^{-1}\left(\mathcal{S}_D^{\alpha,0}\right)^{-1} +O(\omega) & -\frac{\langle \chi_{\p D}, \cdot \rangle}{\pi R^2\left(\omega^2-(\omega^\alpha)^2\right)}\psi_\alpha +O\left(\frac{1}{\omega}\right)
	\end{pmatrix}.$$
	Combining this with equations \eqnref{eq:Mdensity}, \eqnref{eq:Apert} and  \eqnref{eq:E2} yields the desired expression for $M_0(\omega)$.
\end{proof}

\begin{rmk}
	It is clear that $\omega=\omega^\epsilon$ is a characteristic value of $\mathcal{M}^{\epsilon,\delta,\alpha_1}$ if and only if $\omega^\epsilon$ is a characteristic value for $I + M_0$. We have thus reduced the characteristic value problem for the two-dimensional matrix operator $\mathcal{M}^{\epsilon,\delta,\alpha_1}$ to the scalar operator $I+M_0$.
\end{rmk}

\subsection{Defect resonance frequency in the dilute regime}
The following theorem is the main result of this paper. Again, we say a frequency $\omega$ is in the subwavelength regime if $\omega = O(\sqrt{\delta})$.
\begin{thm} \label{thm:dilute}
	For $\delta$ and $R$ small enough, there is a unique characteristic value $\omega^\epsilon(\alpha_1)$ of $\M^{\epsilon,\delta,\alpha_1}(\omega)$ such that $\omega^\epsilon(\alpha_1)\neq \Lambda_{0,\alpha_1}$ and $\omega^\epsilon(\alpha_1)$ is in the subwavelength regime. Moreover, $$\omega^\epsilon(\alpha_1) = \hat\omega+O\left(R^2 + \delta\right),$$ where $\hat\omega$ is the root of the following equation:
	\begin{equation} \label{eq:dilute}
	1 + \left(\frac{\hat\omega^2R^2}{2\delta}\ln\frac{R}{R_d} + \left(1-\frac{R^2}{R_d^2} \right)\right)\frac{1}{2\pi}\int_{Y^*} \frac{(\omega^\alpha)^2}{\hat\omega^2-(\omega^\alpha)^2}\dx \alpha_2 = 0.\end{equation}
\end{thm}
\begin{proof}
We seek the characteristic values of the operator $I+M_0$. We consider the dilute regime, \ie{} where $R$ is small. As shown in \cite{bandgap}, in this case we have
\begin{align}\label{eq:Sdilute}
\S_D^{\alpha,0}[\phi] &= \S_D^0[\phi] + R \mathcal{R}_\alpha(0) \int_{\p D}\phi \dx \sigma +  O(R^2\|\phi\|),
\end{align}
where $\mathcal{R}_\alpha(x) = \Gamma^{\alpha,0}(x) - \Gamma^0(x)$. In particular, 
$$\psi_\alpha = \left(\S_{D}^{\alpha,0}\right)^{-1}[\chi_{\p D}] = -\frac{\capacity_{D,\alpha}}{2\pi R}\chi_{\p D} + O\left(\frac{R^2}{\ln R}\right).$$
We will compute $M_0$ in the Fourier basis. It is known that \cite{MaCMiPaP}
$$(\S_D^0)_{V_m\rightarrow V_n} = -\delta_{m,n}\frac{R}{2|n|}, \quad m\neq 0,$$
which gives
$$\left((\S_{D}^{\alpha,0})^{-1}\right)_{V_m\rightarrow V_n} = -\delta_{m,n}\frac{2|n|}{R} + O(R), \quad m\neq 0.$$
Moreover,
$$\left( \left(P_\alpha^\perp\right)^{-1}\right)_{V_m\rightarrow V_n} =
\begin{cases} \delta_{m,n}\frac{\omega^2}{\omega^2-(\omega^\alpha)^2}+ O\left(\frac{R^2}{\ln R}\right), \quad & n= 0, \\
\delta_{m,n}, & n\neq 0. \end{cases}$$
In total, we have on the subspace $V_0$,
$$(I+M_0)_{V_0\rightarrow V_0} = 1 + \left(\frac{\omega^2R^2}{2\delta}\ln\frac{R}{R_d} + \left(1-\frac{R^2}{R_d^2} \right)\right)\frac{1}{2\pi}\int_{Y^*} \frac{(\omega^\alpha)^2}{\omega^2-(\omega^\alpha)^2}\dx \alpha_2 +O\left(\frac{R^2}{\ln R} + \omega\right).$$
Moreover, if $n\neq 0$, then
$$(I+M_0)_{V_m\rightarrow V_n} = 
\delta_{m,n}\left(\frac{R^{2|n|}}{R_d^{2|n|}}\right) +O\left(R^2 + \omega\right), \quad n\neq 0.
$$
In summary, the operator $I+M_0$ can be written as 
$$ I+M_0(\omega) = \hat M(\omega) + O\left(R^2 + \omega\right),$$
where the limiting operator $\hat M(\omega)$ is a diagonal operator in the Fourier basis, with non-zero diagonal entries for $n\neq 0$. We conclude that $\omega = \hat\omega$ is a characteristic value for $\hat M(\omega)$ if and only if one of the diagonal entries vanishes at $\omega = \hat\omega$, \ie{} if 
\begin{equation}\label{eq:dilutepf}
1 + \left(\frac{\hat\omega^2R^2}{2\delta}\ln\frac{R}{R_d} + \left(1-\frac{R^2}{R_d^2} \right)\right)\frac{1}{2\pi}\int_{Y^*} \frac{(\omega^\alpha)^2}{\hat\omega^2-(\omega^\alpha)^2}\dx \alpha_2 = 0.
\end{equation}
Next, we show that equation \eqnref{eq:dilutepf} has a zero $\hat{\omega} \notin \Lambda_{0,\alpha_1}$ satisfying $\hat{\omega} = O(\sqrt{\delta})$. Introduce the notation
$$I(\omega,\alpha_1) = \frac{1}{2\pi}\int_{Y^*} \frac{(\omega^\alpha)^2}{\omega^2-(\omega^\alpha)^2}\dx \alpha_2,$$
then equation \eqnref{eq:dilutepf} implies
\begin{equation} \label{eq:Iinv}
\hat\omega^2\left(\frac{R^2}{2\delta}\ln\frac{R}{R_d}\right) + \left(1-\frac{R^2}{R_d^2} \right) + \frac{1}{I(\hat{\omega},\alpha_1)} = 0.
\end{equation}
For a fixed $\alpha_1\in Y^*$, define $\omega^* =  \omega^{(\alpha_1,\pi)},$ which is the edge of the first band in $\Lambda_{0,\alpha_1}$. Observe that ${1}/{I(\omega,\alpha_1)}$ is monotonically increasing in $\omega$, and 
$$\lim_{\omega\rightarrow \omega^*}\frac{1}{I(\omega,\alpha_1)} = 0, \qquad \frac{1}{I(\omega,\alpha_1)} \rightarrow \frac{\omega^2}{\omega_0^2} \ \ \text{as} \ \ \omega\rightarrow \infty,$$
where $\omega_0^2$ is the average
$$\omega_0^2 = \frac{1}{2\pi}\int_{Y^*} (\omega^\alpha)^2\dx \alpha_2.$$
In the dilute regime, we can compute 
$$(\omega^*)^2 = -\frac{2\delta}{R^2\ln R} + O\left(\delta^2 + \frac{1}{R}\right),$$
so as $\omega \rightarrow \omega^*$, the right-hand side of equation \eqnref{eq:Iinv} tends to
$$\frac{\ln R_d}{\ln R} - \frac{R^2}{R_d^2} + O\left(\delta + R\right).$$
Since $R_d < R$, the leading-order term is negative. On the other hand, as $\omega \rightarrow \infty$, the right-hand side of \eqnref{eq:Iinv} tends to $\infty$. Since the right-hand side of equation \eqnref{eq:Iinv} is monotonically increasing, this equation has a unique zero $\omega = \hat\omega$. It can be verified that this zero has multiplicity one. Moreover, $\hat{\omega}$ satisfies $\hat\omega = O(\sqrt{\delta})$.

Now, we turn to the full operator $I+M_0(\omega)$. Since $I+M_0(\omega) = \hat M(\omega) + O\left(R^2 + \omega\right),$ by the Gohberg-Sigal theory \cite{MaCMiPaP, AKL, Gohberg1971}, close to $\hat\omega$ there is a unique characteristic value $\omega^\epsilon$ of the operator $I+M_0(\omega)$, satisfying
$$\omega^\epsilon = \hat\omega+O\left(R^2 + \delta\right).$$
This concludes the proof.
\end{proof}

\begin{rmk}
	In the case of $R_d > R$, \ie{} larger defect bubbles, similar arguments show that any subwavelength frequency $\omega^\epsilon(\alpha_1) \in \Lambda_{d,\alpha_1} \setminus \Lambda_{0,\alpha_1}$ satisfies equation \eqnref{eq:dilute} in the dilute regime. However, it is easily verified that this equation has no solutions $\hat\omega > \omega^*$ in the case $R_d>R$. The conclusion is that we must reduce the size of the defect bubbles in order to create subwavelength guided modes in the dilute regime.
\end{rmk}

\subsubsection{Absence of bound modes in the line defect direction}
In this section, we will show that the defect band is not contained in the pure point spectrum of the defect crystal.

\begin{lem} \label{lem:palpha1}
	For $(\alpha_1,\alpha_2) \in Y^*\times Y^*$, $\alpha_2 \neq 0$, the partial derivative of the quasi-periodic Green's function
	$$\frac{\p}{\p \alpha_1}  \Gamma^{\alpha,0}(0)$$
	is zero precisely when $\alpha_1 = 0$ or $\alpha_1 = \pi$.
\end{lem}
\begin{proof}
From the spectral form of the Green's function  \cite{MaCMiPaP}:
$$\Gamma^{\alpha,0}(x) = -\sum_{m\in \Z^2} \frac{e^{i(\alpha+2\pi m)\cdot x}}{|\alpha+2\pi m|^2},$$
it can be easily shown that	
$$\nabla_\alpha \Gamma^{\alpha,0}(0) = \sum_{m\in \Z^2}\frac{\alpha + 2\pi m}{|\alpha + 2\pi m|^4}.$$
By symmetry of the summation, we find that 
$$\frac{\partial}{\partial \alpha_1} \Gamma^{\alpha,0}(0) = 0$$
if and only if $\alpha_1 = 0$ or $\alpha_1=\pi$.
\end{proof}

\begin{prop} \label{prop:palpha1}
	For $\delta$ and $R$ small enough, and for $\alpha_1 \neq 0,\pi$, the characteristic value $\omega^\epsilon = \omega^\epsilon(\alpha_1)$ satisfies 
	$$\frac{\p \omega^\epsilon}{\p \alpha_1} \neq 0.$$
\end{prop}
\begin{proof}
To simplify the computations, we introduce the following notation:
\begin{alignat*}{3}
&a = \frac{R^2}{4\pi\delta}\ln\frac{R}{R_d}, \quad \qquad &&b = \frac{1}{2\pi} \left(1-\frac{R^2}{R_d^2} \right), \\
&x = x(\alpha_1) = \hat{\omega}^2,  &&y = y(\alpha_1,\alpha_2) = (\omega^\alpha)^2.
\end{alignat*}
Then equation \eqnref{eq:dilute} reads
$$\left(ax + b\right)\int_{Y^*} \frac{y}{x-y}\dx \alpha_2 = 1.$$
Denote by $x' = \frac{\p x}{\p \alpha_1}$ and $y' = \frac{\p y}{\p \alpha_1}$, then we have
$$ax' \int_{Y^*} \frac{y}{x-y}\dx \alpha_2 - (ax+b)\int_{Y^*} \frac{x'y - xy'}{(x-y)^2}\alpha_2 = 0,$$
or equivalently, 
$$x'A + B = 0$$
where
$$A = a \int_{Y^*} \frac{y}{x-y}\dx \alpha_2 -(ax+b)\int_{Y^*} \frac{y}{(x-y)^2}\dx\alpha_2,$$
and 
$$B = (ax+b)x \int_{Y^*} \frac{y'}{(x-y)^2}\dx\alpha_2.$$
First, we show that $A\neq 0$ which implies that the zeros of $x'$ coincides with the zeros of $B$. We have
\begin{align*}
A &= a \int_{Y^*} \frac{y}{x-y}\dx \alpha_2 -(ax+b)\int_{Y^*} \frac{y}{(x-y)^2}\dx\alpha_2 \\
&= \int_{Y^*} \frac{ay(x-y)-(ax+b)y}{(x-y)^2}\dx\alpha_2, \\
&= -\int_{Y^*} \frac{y(ay+b)}{(x-y)^2}\dx\alpha_2 < 0,
\end{align*}
since $y(ay+b) > 0$ for all $(\alpha_1,\alpha_2) \in Y^*\times Y^*$.

Next, we show that the leading order of $B$ vanishes exactly at the points $\alpha_1 = 0$ and $\alpha = \pi$. Using equations \eqnref{eq:delta} and \eqnref{eq:Sdilute}, we have 
\begin{align*}
y' &= \frac{\partial}{\partial \alpha_1} (\omega^\alpha)^2 \\
 &= \frac{\partial}{\partial \alpha_1} \left(\frac{-2\delta}{R^2\ln R + 2\pi R^3\mathcal{R}_\alpha(0)}\right) + O\left(\frac{R^3}{\ln R}+ \delta ^2\right)\\
&=\frac{4\pi R^3\delta}{\left(R^2\ln R + 2\pi R^3\mathcal{R}_\alpha(0)\right)^2}\frac{\partial}{\partial \alpha_1}\mathcal{R}_\alpha(0)  + O\left(\frac{R^3}{\ln R}+\delta^2\right).
\end{align*}
Since $\mathcal{R}_\alpha = \Gamma^{\alpha,0} - \Gamma^0$, using Lemma \ref{lem:palpha1} we conclude that for $\delta$ and $R$ small enough and $\alpha_1 \neq 0,\pi$, $y'$ is non-zero for any $\alpha_2$. Hence $B$ is non-zero, which concludes the proof.
\end{proof}

Proposition	\ref{prop:palpha1} shows that the defect dispersion relation is not flat, except for local extrema at $\alpha_1 = 0$ and $\alpha_1 = \pi$. Thus, the defect band is not in the pure point spectrum of the defect operator, and corresponding Bloch modes are not bounded in the line defect direction.

\subsubsection{Bandgap located defect bands}
In this section, we will demonstrate that it is possible to position the entire defect band in the bandgap region with a suitable choice of $\epsilon$. Recall that $\epsilon = R_d - R$. As before, let
$$I(\omega,\alpha_1) = \frac{1}{2\pi}\int_{Y^*} \frac{(\omega^\alpha)^2}{\omega^2-(\omega^\alpha)^2}\dx \alpha_2.$$
\begin{lem} \label{lem:min}
For a fixed $\omega \notin \Lambda_0$, the minimum $$\min_{\alpha_1 \in Y^*} I(\omega,\alpha_1)$$ is attained at $\alpha_1 = \alpha_0$ with $\alpha_0\rightarrow 0$ as $\delta\rightarrow 0$.
\end{lem}
\begin{proof}
	We begin by observing that the minima of $I(\omega,\alpha_1)$ and $\omega^{(\alpha_1,\alpha_2)}$ are attained at the same point $\alpha_1 = \alpha_0\in Y^*$. Using Lemma \ref{prop:palpha1}, for every fixed $\alpha_2 \neq 0$ the minimum of $\capacity_{D,\alpha}$ is attained at $\alpha_1 = 0$, so by Theorem \ref{approx_thm} the minimum of $\omega^{(\alpha_1,\alpha_2)}$ is attained at $\alpha_1 = \alpha_0$ with $\alpha_0 \rightarrow 0$ as $\delta \rightarrow 0$. Since $\omega^{(0,0)} = 0$ (see \cite{bandgap}) this is true for all $\alpha_2 \in Y^*$.
\end{proof}

\begin{prop}
	For $\delta$ small enough, there exists an $\epsilon$ such that for all $\alpha_1 \in Y^*$ we have
	$$\omega^\epsilon(\alpha_1) \notin \Lambda_0.$$
\end{prop}
\begin{proof}
	We want to show that 
	$$\min_{\alpha_1 \in Y^*} \omega^\epsilon(\alpha_1) > \max_{\alpha\in Y^*\times Y^*} \omega^{\alpha}.$$
	Using Lemma \ref{lem:min}, it is easy to see that 
	$$\min_{\alpha_1 \in Y^*} \omega^\epsilon(\alpha_1)$$
	is attained at $\alpha_1 = \alpha_0$. Moreover, from \cite{highfrequency} we know that
	$$\max_{\alpha\in Y^*\times Y^*} \omega^{\alpha}$$
	is attained at $\alpha=\alpha^*=(\pi,\pi)$. Using Theorem \ref{thm:dilute}, we conclude that the lower edge of the defect band coincides with the upper edge of the unperturbed band if 
	$$
	\left(\frac{R^2}{R_d^2} - \frac{\ln R_d}{\ln R} \right) = \frac{1}{I(\omega^*,\alpha_0)} + O\left(\sqrt{\delta} + R\right).
	$$
	For small enough $R$ and $\delta$, the right-hand side is positive, while the left-hand side ranges from $0$ to $+\infty$ for $\epsilon \in (-R,0)$. Hence we can find a solution $\epsilon_0$ to this equation, and the statement holds for $\epsilon > \epsilon_0$.
\end{proof}
\begin{rmk}
	In practice, for $\delta$ small enough, we can approximate $\epsilon_0$ as the root to the equation
	\begin{equation}\label{eq:eps0}
	\left(\frac{R^2}{R_d^2} - \frac{\ln R_d}{\ln R} \right) = \frac{1}{I(\omega^*,0)}.
	\end{equation}
\end{rmk}

\section{Numerical illustrations} \label{sec:num}

\subsection{Implementation}
\subsubsection{Discretization of the operator}
The operator $\mathcal{M}^{\epsilon,\delta,\alpha_1}(\omega)$ was approximated as a matrix $M(\omega)$ using the truncated Fourier basis $e^{-iN\theta}$, $e^{-i(N-1)\theta}, \ldots, e^{iN\theta}$. We refer to \cite{defectSIAM,defectX} for the details of the discretization. The integral over $Y^*$ in \eqnref{eq:Mdensity} was approximated using the trapezoidal rule with $100$ discretization points. The characteristic value problem for $\mathcal{M}^{\epsilon,\delta,\alpha_1}(\omega)$ was formulated as the root-finding problem $\det M(\omega) = 0$ and solved using Muller's method \cite{MaCMiPaP}.

\subsubsection{Evaluation of the asymptotic formula}
The integral over $Y^*$ in equation \eqnref{eq:dilute} was approximated using the trapezoidal rule with $100$ discretization points. Again, the equation was numerically solved using  Muller's method.

\subsection{Dilute regime}
Figure \ref{fig:bandgap_dilute} shows the unperturbed band structure and the defect band for $\alpha_1$ over the Brillouin zone $[0,2\pi]$. The material parameters were chosen as $\kappa_b = \rho_b = 1$, $\kappa_w = \rho_w = 5000$, $R=0.05$ and $\epsilon = -0.2R$. It can be seen that the entire defect band is located inside the deep subwavelength regime of the bandgap. Moreover, the defect frequencies computed using the asymptotic formula agree well with the values computed by discretizing the operator $\mathcal{M}^{\epsilon,\delta,\alpha_1}$. Also, we see that the defect band is not flat. In summary, these results show that the defect crystal supports guided modes in the subwavelength regime, localized to the line defect.

\begin{figure*}[tbh]
	\begin{center}
		\includegraphics[height=5.0cm]{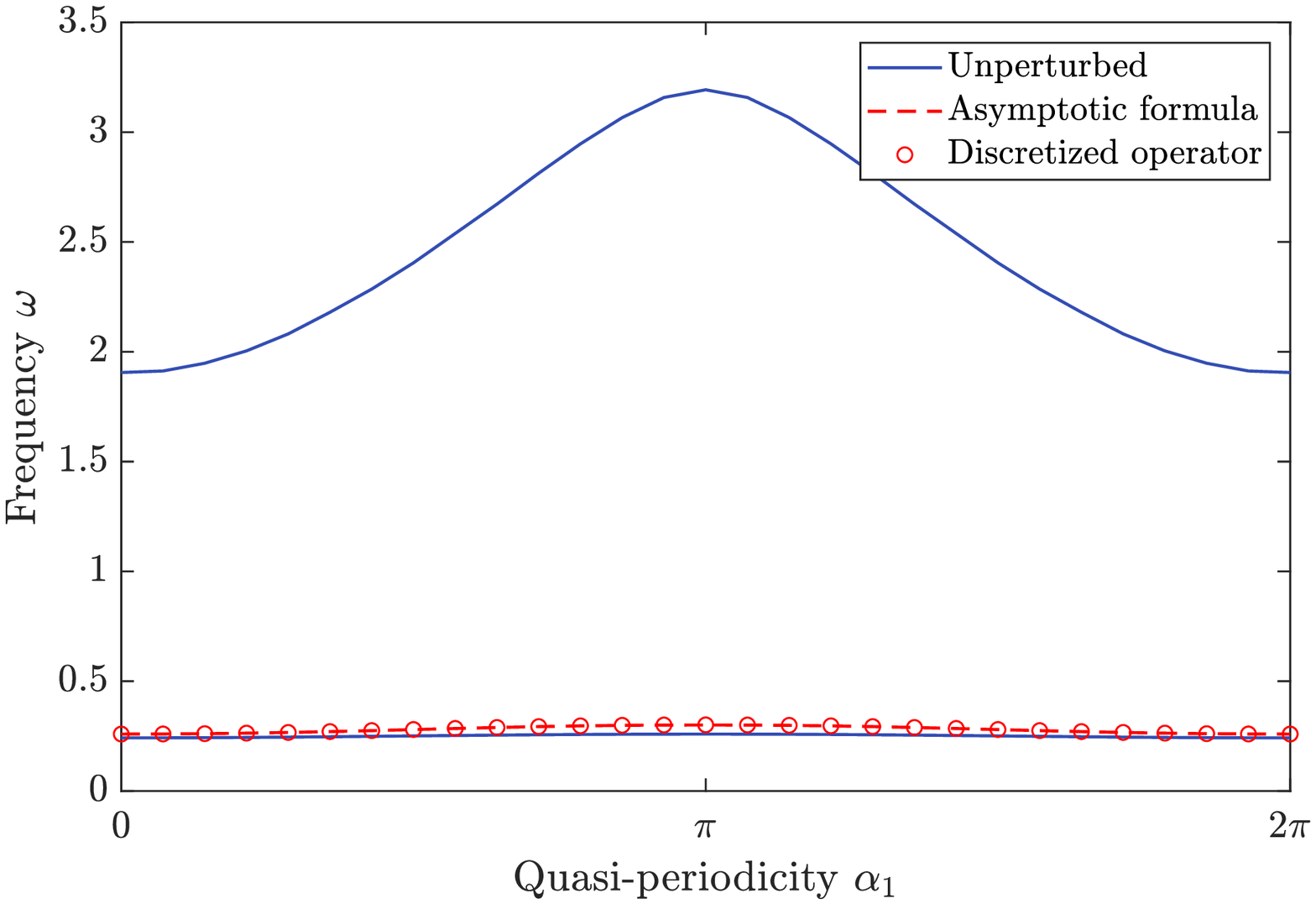}
		\hspace{0.1cm}
		\includegraphics[height=5.0cm]{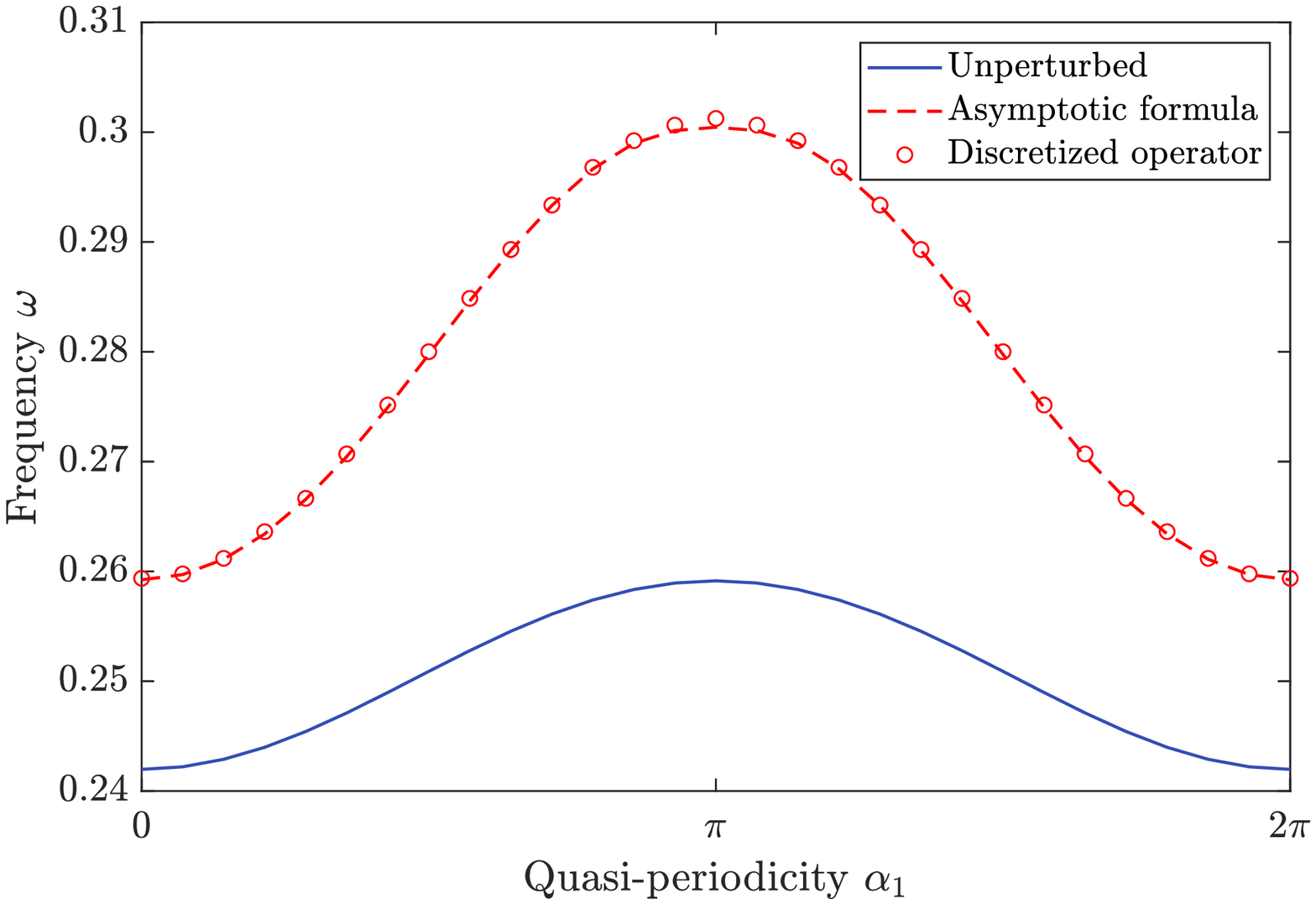}
		\caption{ (Dilute case)
			First two bands of the unperturbed crystal (left) and magnification of the first band and the defect mode (right). The defect band is computed using the asymptotic formula \eqnref{eq:dilute} (red dashed) and by discretizing the operator $\mathcal{M}^{\epsilon,\delta,\alpha_1}$ (red circles). The crystal bubble radius was $R=0.05$ and $\epsilon = -0.2R$.} 
		\label{fig:bandgap_dilute}
	\end{center}
\end{figure*}

\subsubsection{Computation of $\epsilon_0$}
In this section, we numerically compute the critical perturbation size, where the entire defect band is located in the bandgap. The critical perturbation size was computed in two ways: by solving equation \eqnref{eq:eps0} for the leading order term, and by solving the root-finding problem $\omega^{\epsilon_0}(0) = \omega^*$ where $\omega^\epsilon$ was computed by discretizing the operator $\mathcal{M}^{\epsilon,\delta,\alpha_1}$.

Figure \ref{fig:eps0} shows $\epsilon_0$ for different $R$ in the dilute regime.  The material parameters were chosen as $\kappa_b = \rho_b = 1$ and $\kappa_w = \rho_w = 10000$. The values obtained from the asymptotic formula and by discretizing the operator agree, with a smaller radius $R$ giving a smaller error.
Quantitatively, for $R$ in this regime, we require a
decrease of the bubble size by around $14\%$ to $26\%$
in order that the defect band be located inside the
bandgap.

\begin{figure*}[tbh]
	\begin{center}
		\includegraphics[height=5cm]{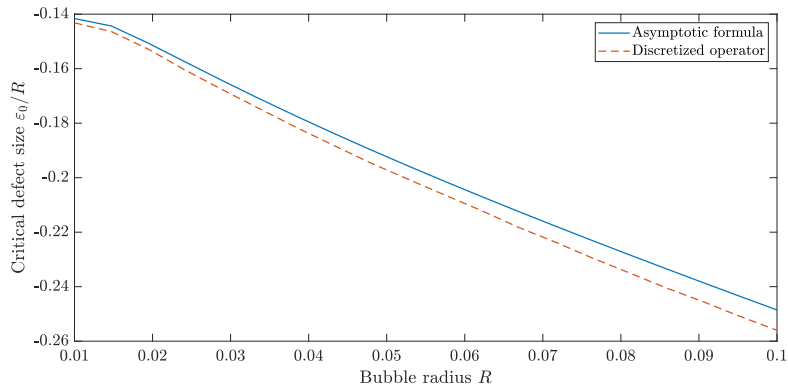}
		\caption{Critical defect size $\epsilon_0$, \ie{} the smallest defect size where corresponding defect band is entirely located inside the bandgap, as a function of the crystal bubble radius.}
		\label{fig:eps0}
	\end{center}
\end{figure*}

\subsection{Non-dilute regime}
Here we compute the defect band in the non-dilute regime, in both cases $\epsilon < 0$ and  $\epsilon > 0$, corresponding to smaller and larger defect bubbles, respectively. Theorem \ref{thm:first} in Appendix \ref{app:nondil} shows that there is a defect frequency $\omega^\epsilon$ in the bandgap for small $\epsilon > 0$ but not for small $\epsilon < 0$. 

\subsubsection{Larger defect bubbles}
Figure \ref{fig:bandgap_ndilute} shows the band structure in the non-dilute case with $\epsilon > 0$. The material parameters were chosen as $\kappa_b = \rho_b = 1$, $\kappa_w = \rho_w = 5000$, $R=0.4$ and $\epsilon = 0.45R$.  As expected from Theorem \ref{thm:first}, there is a defect band above the first band of the unperturbed crystal. Moreover, it is possible to position the entire band inside the bandgap. 

\begin{figure*}[tbh]
	\begin{center}
		\includegraphics[height=5.0cm]{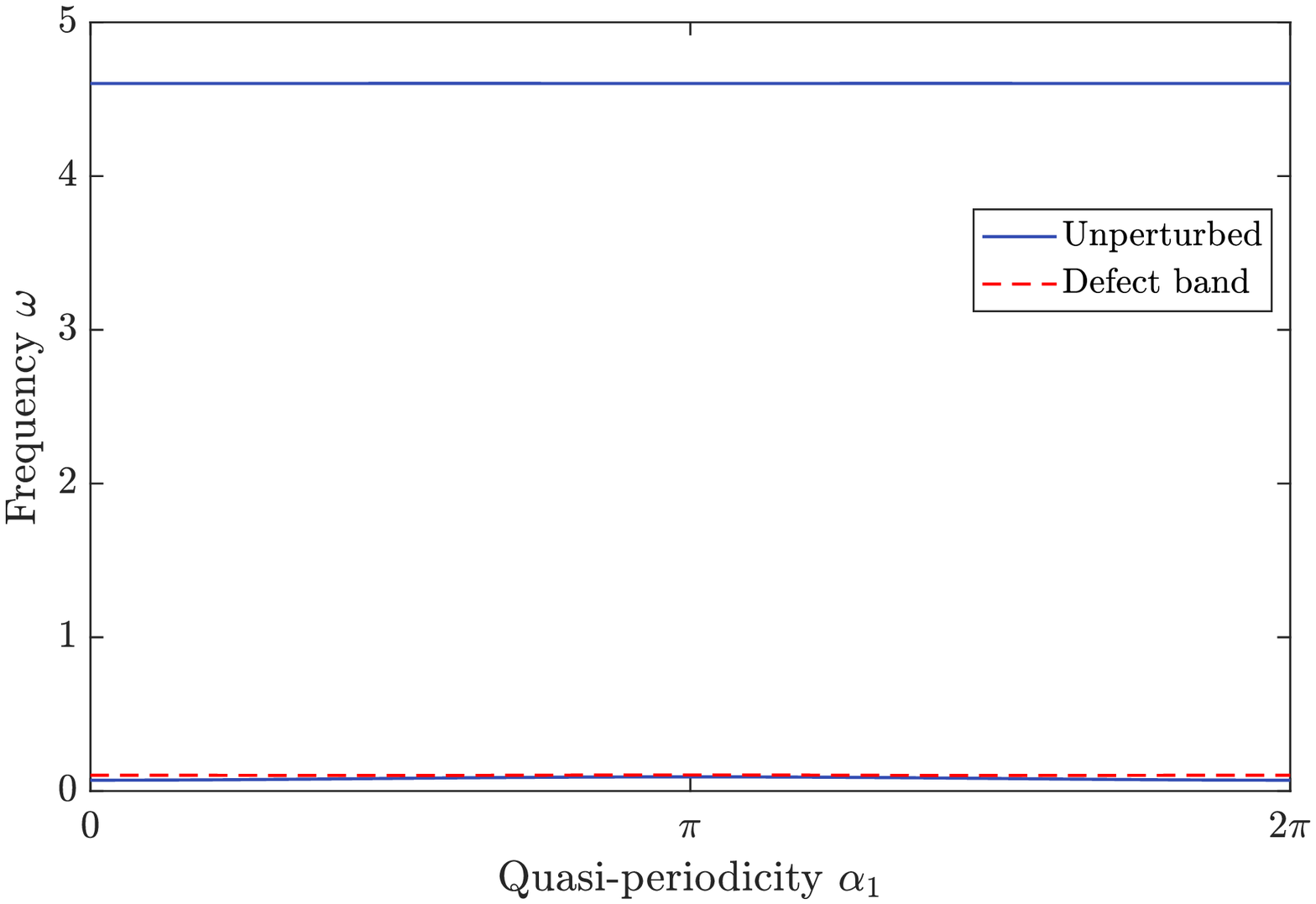}
		\hspace{0.1cm}
		\includegraphics[height=5.0cm]{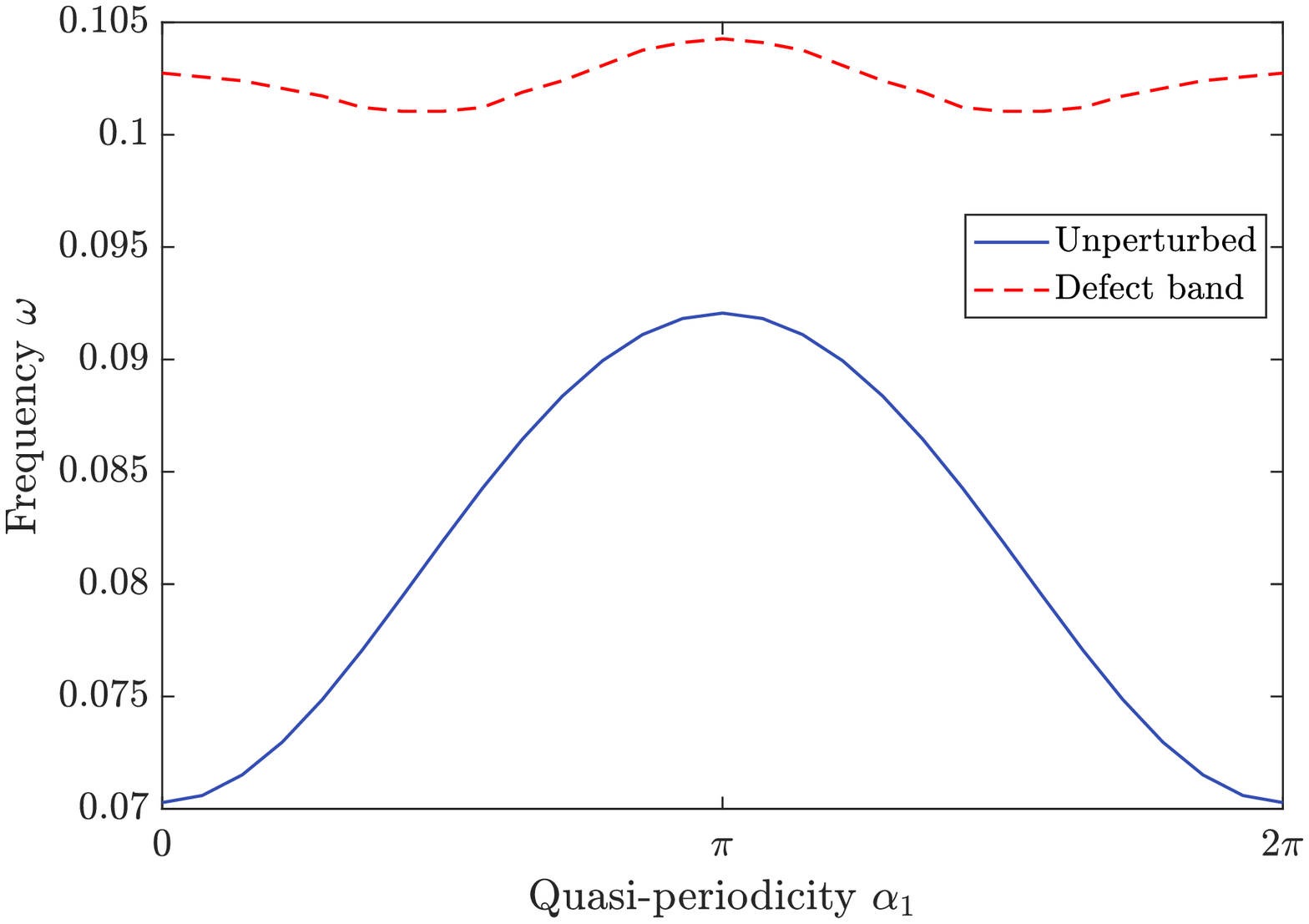}
		\caption{ (Non-dilute case)
			First two bands of the unperturbed crystal (left) and magnification of the first band and the defect mode (right). The defect band was computed by discretizing the operator $\mathcal{M}^{\epsilon,\delta,\alpha_1}$. The crystal bubble radius was $R=0.4$ and $\epsilon = 0.45R$, corresponding to a non-dilute case with larger defect bubbles.} 
		\label{fig:bandgap_ndilute}
	\end{center}
\end{figure*}

\begin{figure*} [tbh]
	\begin{center}
		\includegraphics[height=5.0cm]{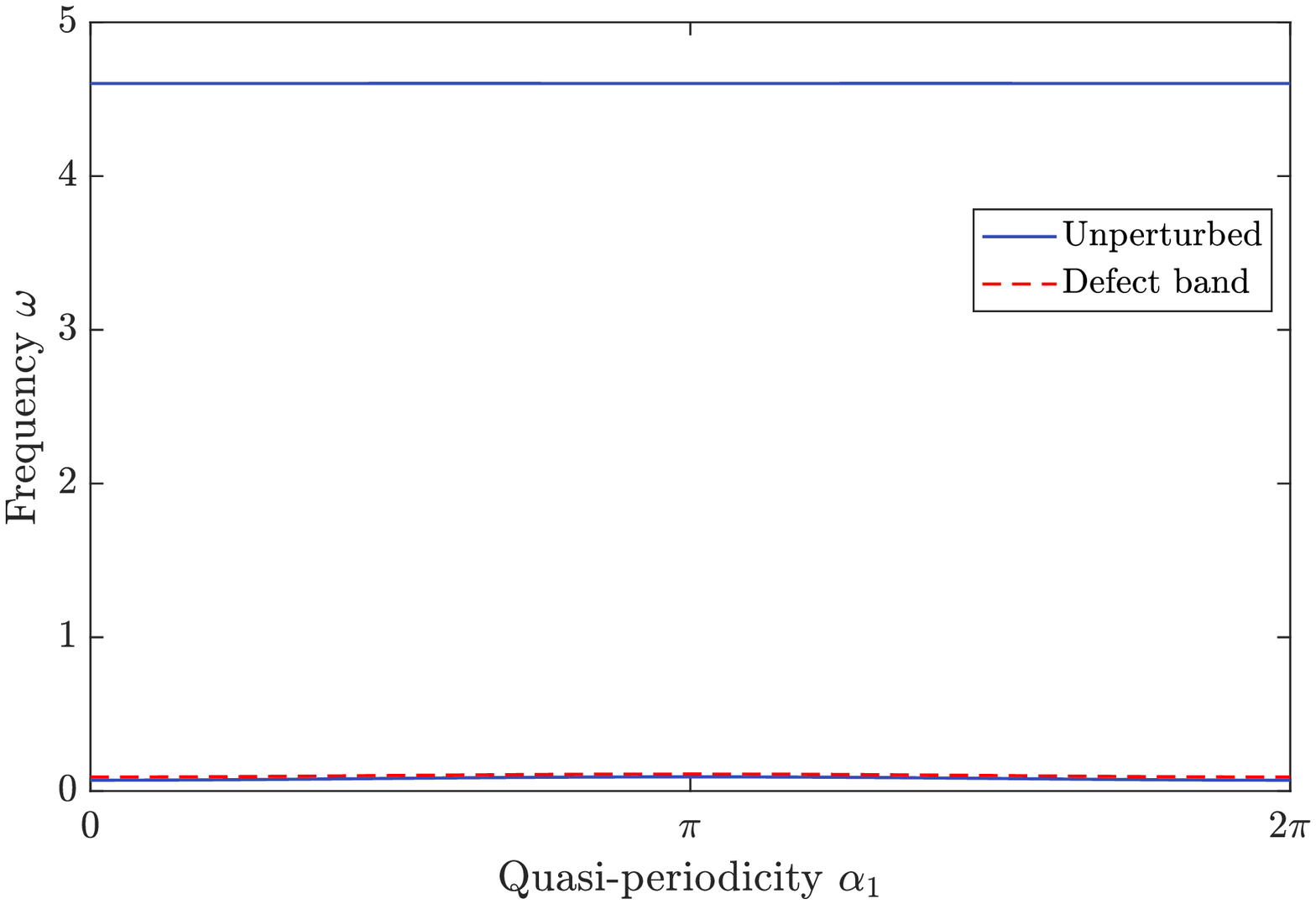}
		\hspace{0.1cm}
		\includegraphics[height=5.0cm]{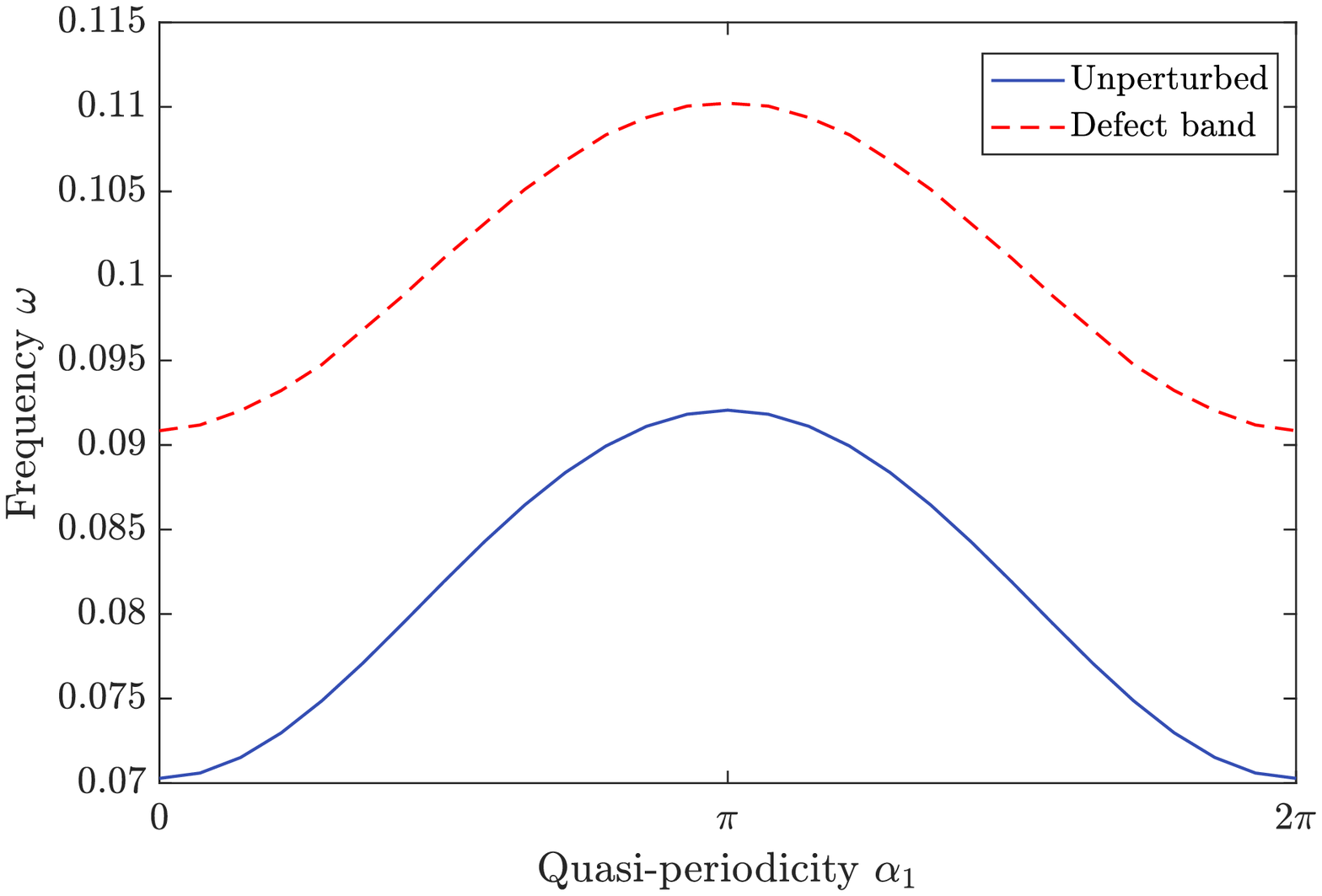}
		\caption{ (Non-dilute case)
			First two bands of the unperturbed crystal (left) and magnification of the first band and the defect mode (right). The defect band was computed by discretizing the operator $\mathcal{M}^{\epsilon,\delta,\alpha_1}$. The crystal bubble radius was $R=0.4$ and $\epsilon = -0.6R$, corresponding to a non-dilute case with smaller defect bubbles.} 
		\label{fig:bandgap_ndilute2}
	\end{center}
\end{figure*}

\subsubsection{Smaller defect bubbles}
Figure \ref{fig:bandgap_ndilute2} shows the band structure in the non-dilute case with $\epsilon = -0.6R$. The material parameters were chosen as $\kappa_b = \rho_b = 1$, $\kappa_w = \rho_w = 5000$, $R=0.4$ and $\epsilon = 0.45R$. In this case a defect band is present inside the bandgap. Here $\epsilon$ is quite large, in contrast to Theorem \ref{thm:first} which is only valid for small $\epsilon$.

\section{Concluding remarks} \label{sec-5}
In this paper, we have for the first time proved the possibility of creating subwavelength guided waves localized to a line defect in a bubbly phononic crystal. We have shown that introducing a defect line, by shrinking the bubbles along the line, creates a defect frequency band inside the bandgap of the original crystal. An arbitrarily small perturbation will create a non-zero overlap between the defect band and the bandgap, and we have explicitly quantified the required defect size in order to position the entire defect band inside the bandgap. Moreover, we have shown for the first time that the defect band is not contained in the pure point spectrum of the perturbed operator. This shows that we can create truly guided modes, which are not localized in the direction of the defect. In the future, we plan to study more sophisticated waveguides, with bends and junctions. Moreover, we also plan to study waveguides in phononic subwavelength bandgap crystals with non-trivial topology, rigorously proving the existence of topologically protected subwavelength states in bubbly crystals.

\appendix

\section{The resonance frequency of the defect mode for small perturbations} \label{app:nondil}
Here we derive a formula for the resonance frequency of the defect mode in the case of small $\epsilon$, following the approach of \cite{defectSIAM,defectX}. The strength of this approach is that it is valid in both the dilute and non-dilute regimes. We begin by reformulating the integral equation \eqnref{eq:Mdensity} in terms of the effective sources $(f, g)$ instead of the layer densities $(\phi,\psi)$. The following proposition is a restatement of Proposition \ref{prop:fg}.

\begin{prop}\label{prop:fg'}
	The effective source pair $(f,g)\in L^2(\partial D)^2$ satisfies the following integral equation:
	\begin{align}\label{eq:fg'}
	\mathcal{M}^{\epsilon,\delta,\alpha_1}(\omega)\begin{pmatrix}
	f
	\\
	g
	\end{pmatrix}:=\bigg(I+(\A_{D}^\epsilon(\omega,\delta) - \A_{D}(\omega,\delta))\frac{1}{2\pi}\int_{Y^*}\A^{(\alpha_1,\alpha_2)}(\omega,\delta)^{-1} \dx \alpha_2\bigg)
	\begin{pmatrix}
	f 
	\\[0.3em]
	g
	\end{pmatrix}
	=
	\begin{pmatrix}
	0
	\\[0.3em]
	0
	\end{pmatrix},
	\end{align}
	for small enough $\delta$ and for $\omega \notin \Lambda_{0,\alpha_1}$ inside the bandgap.
\end{prop}
In this section, we derive an expression for the characteristic value $\omega^\epsilon$ of  $\mathcal{M}^{\epsilon,\delta,\alpha_1}(\omega)$ located slightly above $\omega^\alpha$ for both the dilute and non-dilute regimes.

Let us first analyse the operator $\int_{Y^*} (\A^{\alpha})^{-1}\dx\alpha$.
Since $\omega^\alpha$ is a simple pole of the mapping $\omega \mapsto \A^\alpha(\omega,)^{-1}$ in a neighbourhood of $\omega^\alpha$,
according to \cite{MaCMiPaP}, we can write
\begin{equation} \label{eq:polepencil}
\A^\alpha(\omega)^{-1} = \frac{\L^\alpha}{\omega- \omega^\alpha} + \mathcal{R}^\alpha(\omega),
\end{equation}
where the operator-valued function $\mathcal{R}^\alpha(\omega)$ is holomorphic in a neighbourhood of $\omega^\alpha$, and the operator $\L^\alpha$ maps $L^2(\p D)^2$ onto $\ker \A^\alpha(\omega^\alpha,\delta)$. 
Let us write
$$
\ker {\mathcal{A}^\alpha(\omega^\alpha)} = \mbox{span} \{\Psi^\alpha\}, \quad \ker {\big(\mathcal{A}^\alpha(\omega^\alpha)\big)^*} = \mbox{span} \{\Phi^\alpha\},
$$
where $^*$ denotes the adjoint operator. Then, as in \cite{MaCMiPaP, thinlayer}, it can be shown that 
$$
\L^\alpha = \frac{\langle \Phi^\alpha,\ \cdot\ \rangle \Psi^\alpha}{\langle \Phi^\alpha, \frac{d}{d \omega}\A^\alpha\big|_{\omega=\omega^\alpha} \Psi^\alpha \rangle},
$$
where again $\langle \,\cdot \,,\,\cdot\, \rangle$ stands for the standard inner product of $L^2(\partial D)^2$.

Hence the operator $\mathcal{M}^{\epsilon,\delta,\alpha_1}$  can be decomposed as
$$
\mathcal{M}^{\epsilon,\delta,\alpha_1}(\omega) = I + (\A_{D}^\epsilon - \A_D) \frac{1}{2\pi}\int_{Y^*}\frac{\L^\alpha}{\omega-\omega_\alpha}\dx \alpha_2 +  (\A_{D}^\epsilon - \A_D) \frac{1}{2\pi}\int_{Y^*} \mathcal{R}_\alpha \dx\alpha_2.
$$
Note that the third term in the right-hand side is holomorphic with respect to $\omega$.

Denote by $\alpha^* = (\alpha_1,\pi)$ and $\omega^*(\alpha_1) = \omega^{(\alpha_1,\pi)}$.
Using similar arguments as in \cite{highfrequency} and the fact that each bubble is a circular disk, we can prove the following result on the shape of the dispersion relation close to $\alpha^*$. 
\begin{lem} \label{lem:max_omegaalp}
	For a fixed $\alpha_1 \in Y^*$, the characteristic value $\omega^\alpha$ attains its maximum over $\alpha_2$ at $\alpha_2 = \pi$, \ie{} at $\alpha=\alpha^*$. Moreover, for $\alpha_2$ near $\pi$, we have
	$$
	\omega^\alpha = \omega^*(\alpha_1) - \frac{1}{2}c_\delta(\alpha_1) (\alpha_2-\pi)^2   + o\left((\alpha_2-\pi)^2\right).
	$$
	Here, $c_\delta(\alpha_1)$ is a positive function of $\alpha_1$ and $\delta$.
\end{lem}
The operator $\int_{Y^*}\frac{\L^\alpha}{\omega-\omega^\alpha}\dx \alpha_2$ becomes singular when $\omega\rightarrow \omega^\alpha$. Moreover, since we want to compute the defect band inside the bandgap of the periodic problem at $\alpha_1$, we can assume $\omega$ is inside this bandgap. Consequently, the singularity occurs as $\omega\rightarrow \omega^*$. Let us extract its singular part explicitly. Denote by
$
\A^* = \A^{(\alpha_1,\pi)},  \Phi^* = \Phi^{(\alpha_1,\pi)},$ and $\L^* = \L^{(\alpha_1,\pi)}
$.
Moreover, denote by $B_j$ a bounded function with respect to $\omega$ in $V$. Then, by Lemma \ref{lem:max_omegaalp}, we have
\begin{align*}
\frac{1}{2\pi}\int_{Y^*}\frac{\L^{(\alpha_1,\alpha_2)}}{\omega-\omega^{(\alpha_1,\alpha_2)}} \dx\alpha_2 &= \frac{\L^*}{2\pi}\int_{0}^{2\pi} \frac{1}{\omega - \omega^*+\frac{1}{2} c_\delta(\alpha_1) (\alpha_2-\pi)^2} \dx\alpha_2 + B_1(\omega)
\\
&=\frac{\L^*}{2\pi}\sqrt{\frac{2}{(\omega-\omega^*)c_\delta(\alpha_1)}}2\arctan\left({\sqrt{\frac{c}{2(\omega-\omega^*)}}\pi}\right) + B_1(\omega)\\
&=\frac{\L^*}{\sqrt{2(\omega-\omega^*)c_\delta(\alpha_1)}} + B_2(\omega).
\end{align*}
We therefore get
\begin{align*}
\mathcal{M}^{\epsilon,\delta,\alpha_1}(\omega)=I + \frac{1}{\sqrt{2(\omega-\omega^*)c_\delta(\alpha_1)}} (\A_{D}^\epsilon(\omega^*)-\A_{D}(\omega^*))\mathcal{L}^*  + {\mathcal{R}^\epsilon(\omega)},
\end{align*}
for some $\mathcal{R}^\epsilon(\omega) = O(\epsilon)$ which is analytic and bounded for $\omega$ close to $\omega^*$. We look for characteristic values $\omega= \omega^\epsilon$ of $\mathcal{M}^{\epsilon,\delta,\alpha_1}(\omega)$, \ie{} values such that there exists some $\Psi^\epsilon\neq 0$ with ${\mathcal{M}^{\epsilon,\delta,\alpha_1}(\omega)}\Psi^\epsilon = 0$. Expanding this equation, we have
\begin{align*}
\Psi^\epsilon + \frac{1}{\sqrt{2(\omega^\epsilon-\omega^*)c_\delta(\alpha_1)}} \frac{ (\A_{D}^\epsilon-\A_{D})(\omega^*) \Psi^*}{\langle \Phi^*, \frac{d}{d \omega}\A^*\big|_{\omega=\omega^*} \Psi^* \rangle}\langle \Phi^*, \Psi^\epsilon \rangle + {\mathcal{R}^\epsilon(\omega)}\Psi^\epsilon =0.
\end{align*}
Then, multiplying by $\Phi^*$, we obtain
\begin{align*}
\langle \Phi^*, \Psi^\epsilon \rangle \left(1 + \frac{1}{\sqrt{2(\omega^\epsilon-\omega^*)c_\delta(\alpha_1)}} \frac{ \langle \Phi^*, (\A_{D}^\epsilon-\A_{D})(\omega^*) \Psi^*\rangle}{\langle \Phi^*, \frac{d}{d \omega}\A^*\big|_{\omega=\omega^*} \Psi^* \rangle}\right) + \langle \Phi^*, {\mathcal{R}^\epsilon(\omega)}\Psi^\epsilon\rangle =0.
\end{align*}
Since $\mathcal{R}^\epsilon(\omega) = O(\epsilon)$, it follows from the above equation that  $\langle \Phi^*, \Psi^\epsilon \rangle\neq0$. Therefore, choose $\Psi^\epsilon$ such that $\langle \Phi^*, \Psi^\epsilon \rangle = 1$. This gives
\begin{equation} \label{eq:exact}
\omega^\epsilon = \omega^* + \frac{1}{2c_\delta(\alpha_1)} \frac{1}{\left(1+\langle \Phi^*, {\mathcal{R}^\epsilon(\omega)}\Psi^\epsilon\rangle\right)^2} \left(\frac{\langle \Phi^*, (\A_{D}^\epsilon-\A_{D})(\omega^*) \Psi^*\rangle}{\langle \Phi^*, \frac{d}{d \omega}\A^*\big|_{\omega=\omega^*} \Psi^* \rangle}\right)^2.
\end{equation}
In order to derive a more explicit expression, we will consider the asymptotic limit of $\delta \rightarrow 0$. As in \cite{defectSIAM,defectX}, we have the following lemma.
\begin{lem}\label{lem:estim1} The following results hold:
	\begin{itemize}
		\item[(i)]When $\delta\rightarrow 0$, we have
		$$
		\big\langle \Phi^*, \frac{d}{d \omega}\A^*(\omega^*,\delta) \Psi^* \big\rangle = -2\pi{\omega^*\ln \omega^*} R^3 + O(\sqrt{\delta}),
		$$
		which is positive for $\delta$ small enough.
		\item[(ii)] For a fixed $\epsilon$, when $\delta\rightarrow 0$ we have
		\begin{align*}
		\left\langle\Phi^*,\big(\A_{D}^\epsilon(\omega^*,\delta)-\A_{D}(\omega^*,\delta)\big){ \Psi^*} \right\rangle &=\delta\epsilon \ln\omega^*\left(R\|\psi_{\alpha^*}\|^2_{L^2(\p D)}  - 2\capacity_{D,\alpha^*} \right) + O(\epsilon\delta + \epsilon^2\delta\ln\delta),
		\end{align*}
		where $\psi_{\alpha^*} = (\S_{D}^{\alpha^*,0})^{-1}[\chi_{\p D}]$ and $\capacity_{D,\alpha^*} = -\langle \psi_{\alpha^*}, \chi_{\p D}\rangle$. For small $\epsilon$ and $\delta$, this expression is positive for $\epsilon <0$ if $R$ is small enough, or $\epsilon >0$ for $R$ close enough to $1/2$.
	\end{itemize}
\end{lem}

Combining equation \eqnref{eq:exact} and Lemma \ref{lem:estim1}, we obtain the following result

\begin{thm} \label{thm:first} 	Assume that $\delta$ is small enough and the pair $(R,\epsilon)$ satisfies one of the two assumptions
	\begin{itemize}
		\item[(i)] $R$ small enough and $\epsilon<0$ small enough in magnitude (Dilute regime),
		\item[(ii)] $R$ close enough to $1/2$ and $\epsilon>0$ small enough (Non-dilute regime).
	\end{itemize} 
	Then there exists one frequency value $\omega^\epsilon(\alpha_1)$ such that the problem \eqref{eq:scattering-strip} has a non-trivial solution and $\omega^\epsilon(\alpha_1)$ is slightly above $\omega^*(\alpha_1)$. Moreover, as $\delta \rightarrow 0$ we have
	\begin{equation} \label{eq:defect-freq}
	\omega^\epsilon(\alpha_1) = \omega^*(\alpha_1) + \frac{1}{2c_\delta(\alpha_1)}\left(\frac{\delta\epsilon\left(R\|\psi_{\alpha^*}\|^2_{L^2(\p D)}  - 2\capacity_{D,\alpha^*} \right)}{2\pi{\omega^*(\alpha_1)} R^3}\right)^2 + O\left(\frac{\epsilon^2\sqrt{\delta}}{\ln\delta} + \epsilon^3\sqrt{\delta}\right),
	\end{equation}
	where $\alpha^* = (\alpha_1, \pi)$.
\end{thm}

\begin{rmk}
	It is easily verified that equation \eqnref{eq:dilute} evaluated for small $\epsilon$ coincides with equation \eqnref{eq:defect-freq} evaluated for small $R$.
\end{rmk}

\section{Characterization of the effective sources for non-circular bubbles} \label{app:noncirc}
Let now $D$ be a general simply connected domain with $\partial D \in C^1$. In this section, we will restrict to the case of small size perturbations $\epsilon <0 $. Define the defect bubble $D_d \in D$ as the domain with boundary

$$ \partial D_d = \{x + \epsilon\nu_x | x\in \partial D \},$$
where $\nu_x$ is the outward unit normal of $\partial D$ at $x\in\partial D$. We will need some results given in \cite{thinlayer}. First, we introduce some notation.
Define the mapping $p: \partial D  \rightarrow \partial D_d,  \ p(x) = x+\epsilon  \nu_x$. Let $x,y\in \partial D$ and let $\widetilde{x} = p(x) \in \partial D_d$ and $\widetilde{y} = p(y) \in \partial D_d$. Define $q: L^2( \partial D) \rightarrow L^2(\partial D_d), \ q(\phi)(\widetilde x) = \phi ( p^{-1}(\widetilde x)) $, and for a surface density $\phi$ on $\partial D$, define $\widetilde{\phi} = q(\phi)$ on $\partial D_d$. 

We also define the signed curvature $\tau = \tau(x), x\in \partial D$ in the following way. Let $x = x(t)$ be a parametrization of $\partial D$ by arc length. Then define $\tau$ by
$$
\frac{d^2}{dt^2}x(t) = -\tau \nu_x.
$$
Observe that $\tau$ is independent of the orientation of $\partial D$. The following results are given in \cite{thinlayer}, but adjusted to the case where $\epsilon <0$.
\begin{prop}\label{prop:asympsingle}
	Let $k>0$. 	Let $\phi \in L^2(\partial D)$ and let $x,y,\widetilde{x},\widetilde{y},\widetilde{\phi}$ be as above. Then 
	\begin{equation} \label{eq:asympSdD}
	\S_{D_d,D}^{k}[\phi](\widetilde{x}) = \S_D^k[\phi](x) +\epsilon \left(-\frac{1}{2}I + \K_D^{k,*}\right)[\phi](x) + o(\epsilon),
	\end{equation}
	\begin{equation} \label{eq:asympSd}
	\S_{D_d}^k[\widetilde{\phi}](\widetilde{x}) = \S_D^k[\phi](x) + \epsilon \left(\K_D^k + \K_D^{k,*}\right)[\phi](x) + \epsilon\S_D^k[\tau\phi](x) + o(\epsilon),
	\end{equation}
	\begin{equation} \label{eq:asympSDd}
	\S_{D,D_d}^k[\widetilde{\phi}](x) = \S_D^k[\phi](x) + \epsilon \left(-\frac{1}{2}I+ \K_D^k \right)[\phi](x) + \epsilon\S_D^k[\tau\phi](x) + o(\epsilon).
	\end{equation}
\end{prop}

\begin{prop} \label{prop:asympK}
	Let $\phi \in L^2(\partial D)$ and let $x,y,\widetilde{x},\widetilde{y},\widetilde{\phi}$ be as above. Then 
	\begin{equation} \label{eq:asympK}
	\K_{D_d}^{k,*}[\widetilde{\phi}](\widetilde{x}) = \K_D^{k,*}[\phi](x) + \epsilon\K_1^k[\phi](x) + o(\epsilon),
	\end{equation}
	where $\K_1^k$ is given by
	\begin{equation*}
	\K_1^k = \K_D^{k,*}[\tau\phi](x) - \tau(x) \K_D^{k,*}[\phi](x)  + \frac{\partial \D_D^k}{\partial \nu}[\phi](x) - \frac{\partial^2}{\partial T^2}\S_D^k[\phi](x) - k^2\S_D^k[\phi](x).
	\end{equation*}
	Here $\frac{\partial^2}{\partial T^2}$ denotes the second tangential derivative, which is independent of the orientation of $\partial D$. 
\end{prop}

We also state the following result which is given, for example, in \cite{McLean}.
\begin{prop} \label{prop:hypersingular}
	For $x\in \partial D$ and $k\geq 0$ we have
	\begin{align*}
	\frac{\partial \D_D^k}{\partial \nu}[\phi](x) &= \left(\frac{1}{2}I + \K_D^{k,*}\right)\left(\S_D^k\right)^{-1}\left(-\frac{1}{2} + \K_D^k\right)[\phi](x) \\ 
	&= \left(-\frac{1}{2}I + \K_D^{k,*}\right)\left(\S_D^k\right)^{-1}\left(\frac{1}{2} + \K_D^k\right)[\phi](x).
	\end{align*}
\end{prop}

As in Section \ref{sec-3}, we consider the defect problem \eqnref{eq:scattering}, modelled by the fictitious sources as in equation \eqnref{eq:scattering_fictitious}. Observe that Proposition \ref{prop:floquet} is valid even for the case of non-circular bubbles. To derive the analogue of Proposition \ref{prop:effective}, we again study equations \eqnref{eq:SDdSD_inside} and \eqnref{eq:SDdSD_outside}, \ie{},   
\begin{align*}
\mathcal{S}_{D_d}^{k_b}[\varphi_d] &\equiv \mathcal{S}_D^{k_b}[\varphi] \quad\mbox{in }D_d, \\
\mathcal{S}_{D_d}^{k_w}[\psi_d] &\equiv \mathcal{S}_D^{k_w}[\psi] \quad\mbox{in }Y^2\setminus \overline{D}.
\end{align*}
Since $\omega$ is in the subwavelength regime, $k_b$ is not a Dirichlet eigenvalue. Together with the uniqueness of the exterior Dirichlet problem, we conclude that it is sufficient to consider these equations on the boundaries. Using the notation from above, this means
\begin{align*}
\mathcal{S}_{D_d}^{k_b}[\varphi_d] &= \mathcal{S}_{D_d,D}^{k_b}[\varphi] ,\\
\mathcal{S}_{D,D_d}^{k_w}[\psi_d] &= \mathcal{S}_D^{k_w}[\psi].
\end{align*}
Using the expansions \eqnref{eq:asympSdD}, \eqnref{eq:asympSd} and \eqnref{eq:asympSDd}, we find
\begin{align*}
\S_D^{k_b}[q^{-1}\varphi_d] + \epsilon \left(\K_D^{k_b} + \K_D^{k_b,*}\right)[q^{-1}\varphi_d] + \epsilon\S_D^{k_b}[\tau q^{-1}\varphi_d] &=\S_D^{k_b}[\varphi] +\epsilon \left(-\frac{1}{2}I + \K_D^{k_b,*}\right)[\varphi] + o(\epsilon)\\  
\S_D^{k_w}[q^{-1}\psi_d](x) + \epsilon \left(-\frac{1}{2}I+ \K_D^{k_w} \right)[q^{-1}\psi_d](x) + \epsilon\S_D^{k_w}[\tau q^{-1}\psi_d](x) &= \mathcal{S}_D^{k_w}[\psi] + o(\epsilon),
\end{align*}
with $q$ defined as above. From this we find that
\begin{align} \label{eq:noncirc1}
\begin{pmatrix}
\varphi_d \\
\psi_d
\end{pmatrix} =& Q\left(I + \epsilon\begin{pmatrix}
- \left(\S_D^{k_b}\right)^{-1}\left(\frac{1}{2}I + \K_D^{k_b} \right) - \tau & 0 \\ 0 & - \left(\S_D^{k_w}\right)^{-1}\left(-\frac{1}{2}I + \K_D^{k_w} \right) - \tau
\end{pmatrix}\right)\begin{pmatrix}
\varphi \\
\psi
\end{pmatrix} + o(\epsilon) \nonumber \\ 
:=& \mathcal{P}_1\begin{pmatrix}
\varphi \\
\psi
\end{pmatrix},
\end{align}
where $Q$ is the bijection $Q: \left(L^2(\partial D)\right)^2 \rightarrow \left(L^2(\partial D_d)\right)^2, Q = (q, q)$ and $\mathcal{P}_1: \left(L^2(\partial D)\right)^2 \rightarrow \left(L^2(\partial D_d)\right)^2$.

Using the asymptotic expansions \eqnref{eq:asympSd} and \eqnref{eq:asympK}, we can expand the operator $\A_{D_d}$ as 
\begin{equation} \label{eq:noncirc2}
\A_{D_d} = Q \circ \left(\A_{D}(\omega,\delta) + \epsilon \A_1(\omega,\delta)\right)\circ Q^{-1}  + o(\epsilon),
\end{equation}
where
\begin{equation*} \label{eq:A1}
\A_1(\omega,\delta) = 
\begin{pmatrix}
\K_D^{k_b} + \K_D^{k_b,*} + \S_D^{k_b}[\tau\cdot] & -\left(\K_D^{k_w} + \K_D^{k_w,*} + \S_D^{k_w}[\tau\cdot]\right)\\
\K_1^{k_b} & -\delta \K_1^{k_w}
\end{pmatrix}.
\end{equation*}

Using Taylor expansion, we have that 
$$\frac{\partial}{\partial \nu}  H |_{\partial D} = \frac{\partial}{\partial \nu}  H |_{\partial D_d} - \epsilon \frac{\partial^2}{\partial \nu^2}H |_{\partial D_d} .$$
We use the Laplacian in the curvilinear coordinates defined by $T_{\widetilde x},\nu_{\widetilde x}$ for $\widetilde x\in \partial D_d$,
\begin{equation*}\label{eq:lapcurve}
\Delta = \frac{\partial^2}{\partial \nu^2} + \tau(\widetilde x)\frac{\partial}{\partial \nu} + \frac{\partial^2}{\partial T^2}.
\end{equation*}
It is easily verified that the curvatures on the two boundaries satisfy $$\tau(\widetilde x) = \tau(x) +O(\epsilon).$$
Hence we obtain
$$ \frac{\partial^2}{\partial \nu^2}H |_{\partial D_d} = -\left(k_w^2 +\frac{\partial^2}{\partial T^2}\right) H |_{\partial D_d} - \tau \frac{\partial}{\partial \nu}  H |_{\partial D_d} + O(\epsilon).$$
In total, we have
\begin{equation} \label{eq:noncirc3}
\begin{pmatrix}
H|_{\partial D}
\\[0.3em]
\ds\partial H /\partial \nu |_{\partial D}
\end{pmatrix}
=  \mathcal{P}_2^{-1}
\begin{pmatrix}
H|_{\partial D_d}
\\[0.3em]
\ds\partial H /\partial \nu |_{\partial D_d}
\end{pmatrix},
\end{equation}
where the operator $\mathcal{P}_2^{-1}: L^2(\p D_d)^2\rightarrow L^2(\p D)^2 $ is given by
$$
\mathcal{P}_2^{-1} = \left( I + \epsilon\begin{pmatrix}
\ds 0
& -1
\\
k_w^2 +\left(\partial_T\right)^2& \tau
\end{pmatrix}\right)Q^{-1} + o(\epsilon).
$$

Combining equations \eqnref{eq:AD}, \eqnref{eq:ADd_original}, \eqnref{eq:noncirc1} and \eqnref{eq:noncirc3} we arrive at
\begin{equation*}
\left(\mathcal{P}_2^{-1}\A_{D_d}\mathcal{P}_1 - \A_{D}\right)\begin{pmatrix}\varphi \\ \psi \end{pmatrix} = \begin{pmatrix}f \\ g \end{pmatrix}.
\end{equation*}
As before, we define $\mathcal{A}_D^\epsilon = \mathcal{P}_2^{-1}\A_{D_d}\mathcal{P}_1$. Finally, we can compute this explicitly using equations \eqnref{eq:noncirc1}, \eqnref{eq:noncirc2} and \eqnref{eq:noncirc3} and Proposition \ref{prop:hypersingular} to obtain the following proposition, which is the analogue of Proposition \ref{prop:effective} in the case of non-circular bubbles. 
\begin{prop} \label{prop:effective_noncirc}
	The density pair $(\varphi,\psi)$ and the effective sources $(f,g)$ satisfy the following relation 
	\begin{equation*}
	\left(\mathcal{A}_D^\epsilon - \A_{D}\right)\begin{pmatrix}
	\varphi
	\\[0.3em]
	\psi
	\end{pmatrix} = 
	\begin{pmatrix}
	f
	\\[0.3em]
	g
	\end{pmatrix},
	\end{equation*}
	where the operator $\mathcal{A}_D^\epsilon$ satisfies
	$$ \mathcal{A}_D^\epsilon - \A_{D} = \epsilon(\delta-1)\begin{pmatrix}
	0 & \frac{1}{2}I + \K_D^{\omega,*}
	\\[0.3em]
	0 & \omega^2 \S_D^\omega + \frac{\partial^2}{\partial T^2}\S_D^\omega
	\end{pmatrix} + o(\epsilon) .$$
\end{prop}

\bibliography{line}{}
\bibliographystyle{abbrv}

\end{document}